\documentclass[final]{siamltex}
\usepackage{multirow}
\usepackage{url}
\usepackage{showkeys}
\usepackage{amsmath,amssymb,epsfig}

\newtheorem{rmk}{Remark}[section]

\newcommand{\nwc}{\newcommand}

\nwc{\qref}[1]{(\ref{#1})} 

\newcommand{\tr}{\mbox{tr}}
\renewcommand{\div}{\nabla\!\cdot\!}
\nwc{\ip}[1]{\langle #1 \rangle}

\newcommand{\RR}{\mathbb{R}}
\newcommand{\PP}{\mathbb{P}}
\newcommand{\Omg}{\Omega}
\newcommand{\Gam}{\Gamma}
\newcommand{\eps}{\mbox{\boldmath $\epsilon$}}
\renewcommand{\gg}{\mbox{\boldmath $g$}}
\renewcommand{\ss}{\mbox{\boldmath $s$}}
\newcommand{\ww}{\mbox{\boldmath $w$}}
\newcommand{\ee}{\mbox{\boldmath $e$}}

\renewcommand{\aa}{\mbox{\boldmath $a$}}

\newcommand{\nn}{\mbox{\boldmath $n$}}
\newcommand{\ff}{\mbox{\boldmath $f$}}
\newcommand{\FF}{\mbox{\boldmath $F$}}
\newcommand{\GG}{\mbox{\boldmath $G$}}
\newcommand{\HH}{\mbox{\boldmath $H$}}
\newcommand{\EE}{\mbox{\boldmath $E$}}

\newcommand{\tuu}{\mbox{\boldmath $\tilde{u}$}}
\newcommand{\tvv}{\mbox{\boldmath $\tilde{v}$}}
\newcommand{\uu}{\mbox{\boldmath $u$}}
\newcommand{\sig}{\mbox{\boldmath $\sigma$}}
\newcommand{\eeta}{\mbox{\boldmath $\eta$}}

\newcommand{\II}{\mbox{\boldmath $I$}}
\newcommand{\vv}{\mbox{\boldmath $v$}}

\newcommand{\xx}{\mbox{\boldmath $x$}}
\newcommand{\zz}{\mbox{\boldmath $z$}}
\newcommand{\yy}{\mbox{\boldmath $y$}}

\renewcommand{\grad}{\nabla}
\newcommand{\pa}{\partial}
\newcommand{\lam}{\lambda}

\newcommand{\TT}{\mathcal{T}}
\newcommand{\Dt}{\Delta t}
\renewcommand{\div}{\nabla \cdot}

\renewcommand{\(}{\left(}
\renewcommand{\)}{\right)}
\newcommand{\<}{\left<}
\renewcommand{\>}{\right>}
\nwc{\ta}{\tilde{a}}
\renewcommand{\t}{{(t)}}

\newcommand{\np}{{n\!-\!\frac12}}
\newcommand{\npp}{{n\!-\!1}}

\newcommand{\vphi}{\mbox{\boldmath $\varphi$}}
\newcommand{\vpsi}{\mbox{\boldmath $\psi$}}
\newcommand{\tphi}{\mbox{\boldmath $\phi$}}

\newcommand{\vecw}{\mbox{\reflectbox{$\vec{\reflectbox{\ww}}$}}}
\newcommand{\vecu}{\mbox{\reflectbox{$\vec{\reflectbox{\uu}}$}}}
\newcommand{\vecf}{\mbox{\reflectbox{$\vec{\reflectbox{\ff}}$}}}
\newcommand{\lvecu}{\vec{\uu}}

\newcommand{\vecphi}{\mbox{\reflectbox{$\vec{\reflectbox{$\phi$}}$}}}

\newcommand{\vectphi}{\mbox{\reflectbox{$\vec{\reflectbox{\tphi}}$}}}

\newcommand{\vpi}{\mbox{\boldmath $\pi$}}
\newcommand{\IIs}{{I\!I_s}}

\title{Combined field formulation and a simple stable explicit interface advancing scheme for fluid structure interaction}
\author{Jie Liu\thanks{Department of Mathematics, National University
of Singapore, Singapore 119076 ({\tt {matlj@nus.edu.sg}})} }
\begin{document}
\maketitle

\begin{abstract}
We develop a combined field formulation for the fluid structure (FS)
interaction problem. The unknowns are $(\uu;p;\vv)$, being the fluid
velocity, fluid pressure and solid velocity. This combined field
formulation uses Arbitrary Lagrangian Eulerian (ALE) description for
the fluid and Lagrangian description for the solid. It automatically
enforces the simultaneous continuities of both velocity and traction
along the FS interface. We present a first order in time fully
discrete scheme when the flow is incompressible Navier-Stokes and
when the solid is elastic. The interface position is determined by
first order extrapolation so that the generation of the fluid mesh
and the computation of $(\uu;p;\vv)$ are decoupled. This explicit
interface advancing enables us to save half of the unknowns
comparing with traditional monolithic schemes. When the solid has
convex strain energy (e.g. linear elastic), we prove that the total
energy of the fluid and the solid at time $t^{n}$
is bounded by the total energy at time $t^{n-1}$. 
Like in the continuous case, the fluid mesh velocity
which is used in ALE description does not enter into the stability bound.
{\em Surprisingly, the nonlinear convection term in the Navier-Stokes equations
plays a crucial role to stabilize the scheme} and the stability result does not apply to Stokes flow.
As the nonlinear convection term is treated semi-implicitly,
in each time step, we only need to solve a linear system (and only once) which involves merely $(\uu;p;\vv)$ if
the solid is linear elastic.
Two numerical tests including the benchmark test of Navier-Stokes flow past a
Saint Venant-Kirchhoff elastic bar are performed. In addition to
the stability, we also confirm the first order
temporal accuracy of our explicit interface advancing scheme.
\end{abstract}

\begin{keywords}
Fluid Structure Interaction; Arbitrary Lagrangian Eulerian;
Navier-Stokes Equations; Saint Venant-Kirchhoff;
\end{keywords}

\section{Introduction}
The interaction of a rigid or deformable solid with its surrounding fluid or
the fluid it enclosed gives rise to a
very rich variety of phenomena. For example, a rotating fan, a vibrating aircraft wing,
a swimming fish. The daily activities of many parts of our human body are also
intimately related to this fluid solid or more commonly called fluid structure (FS) interaction.
For example, the heart beating, respiration,
speaking, hearing, and even snoring.
To understand those phenomena, we need to model both the fluid and the
solid. In this paper, we assume the fluid is incompressible and the solid
is deformable. The governing equations for FS interaction are then as follows:
\begin{align}
     \rho^f (\pa_t \uu+\uu\cdot\grad\uu) = \div \sig^f + \rho^f\gg^f,
     \qquad \div\uu=0 \qquad\mbox{ in }\Omg^f_{(t)},    \label{dn1.1}\\
     \rho^s \pa_{tt} \vphi = \div \sig^s + \rho^s\gg^s \qquad\mbox{ in }\Omg^s \label{dn1.2}.
\end{align}
In the above equations, $\Omg^f_\t \subset \RR^d$ is the fluid domain at time $t$ and
$\Omg^s \subset \RR^d$ is the initial configuration of the solid.
$\uu=\uu(\xx,t)$ is the velocity at a spatial point
$\xx$ in $\Omg^f_\t$.
$\vphi=\vphi(\zz,t)$ is
the position at time $t$ of a material point $\zz$ in $\Omg^s$.
Constants $\rho^f$ and $\rho^s$ are the densities of the fluid and the solid.
As one can tell, the fluid is described by the Eulerian (spatial)
description while the solid is described by the Lagrangian
(material) description with reference configuration $\Omg^s$.
$\sig^f$ is the stress of the fluid in the Eulerian
description and $\sig^s$ is the stress of the solid in the Lagrangian
description. If the fluid is assumed to be viscous
Newtonian,
\begin{equation}
 \sig^f=\sig^f(\uu,p)=2\rho^f\nu^f\,\eps(\uu)-p\,I=\rho^f\nu^f(\grad\uu+\grad\uu^\top)-p\,I.
\end{equation}
If the solid is elastic, then there is a strain energy $I_s$ with density $W$:
\begin{equation} \label{strain}
 I_s(\vphi(\cdot,t))=\int_{\Omg^s} W(\grad \vphi(\zz,t)) d\zz,
\end{equation}
so that the stress $\sig^s$ is determined by $\sig^s(\FF)=\pa W(\FF)/\pa \FF$.
Here $\FF(\zz,t)=\grad\vphi(\zz,t)$ is the deformation gradient.
As a result, $\sig^s$ satisfies
\begin{equation} \label{strain.1}
 \left.\frac{d}{d\epsilon}\right|_{\epsilon=0}I_s(\vphi+\epsilon \tphi)=\<\sig^s(\vphi),\grad\tphi\>_{\Omg^s}
 \qquad \forall \; \vphi \; \mbox{and} \; \tphi.
\end{equation}
Like in $\sig^f(\uu,p)$, we write $\sig^s(\vphi)$ instead of $\sig^s(\grad\vphi)$ for simplicity.
Here $\<A,B\>_{\Omg^s}=\int_{\Omg_s} (A,B)dx$ and $(A,B)=A:B=\mbox{tr}(A^\top B)$. 
Let $\EE=\frac12 \(\FF^\top \FF - \II\)$ be the strain tensor. If the strain of the solid is small,
we can model the solid as Saint Venant-Kirchhoff material with
$W=W_{SV\!K}(\FF)=\mu^s \tr(\EE^2) + \frac{\lam^s}{2}(\tr\EE)^2 $. Simple calculation shows
\begin{equation} \label{svk.constitutive}
 \sig^s(\vphi)= 2\mu^s \FF \EE + \lam^s (\mbox{tr} \EE) \FF .
\end{equation}
If the deformation of the solid is small, one popular choice of
$\sig^s$ is
\begin{equation} \label{lin.material}
 \sig^s(\vphi)=\mu^s(\grad\eeta+\grad\eeta^\top)+\lam^s(\div\eeta)\II,
\end{equation}
where $\eeta(\zz,t)=\vphi(\zz,t)-\zz$ is the displacement vector.
We can define an associated $W=W_L(\FF)=\mu^s\tr\(\EE_L^2\) + \frac{\lam^s}{2} \(\tr \EE_L\)^2$
with $\EE_L=\frac{\FF+\FF^\top}{2}-\II$ and check
$\(\frac{\pa W_L\(\FF\)}{\pa \FF},\GG\) = \mu^s\(\(\FF-\II\)+\(\FF-\II\)^\top, \GG\)+\lam^s\(\tr(\FF-\II)\II, \GG\)$.
We can also verify the convexity of $W_L$:
\[
\(\frac{\pa^2 W_L\(\FF\)}{\pa \FF^2}:\GG,\GG\)
\!\!=\!\!\(\left.\frac{d}{d\epsilon}\right|_{\epsilon=0}\!\!\!\frac{\pa W_L\(\FF\!+\!\epsilon \GG\)}{\pa \FF},\GG\)
\!\!=\!\!\frac{\mu^s}{2}\(\GG\!+\!\GG^\top, \GG\!+\!\GG^\top\!\)+\lam^s(\tr\GG)^2.
\]
But we will see later that $W_{SV\!K}$ is not convex (see \qref{svk.As}).
Indeed, it is not even polyconvex \cite[sec 3.9]{Ci2}


The fluid and the solid
are coupled by the continuity of velocity and traction across the FS
interface. Mathematically, it can be written as (see Antman \cite[p.485,
$(15.2b)$, p.489, (15.31)]{An})
\begin{align}
     \uu(\vphi(\zz,t),t)=\pa_t \vphi(\zz,t) \qquad \forall \zz\in \Gam, \label{dn1.3} \\
     \int_{\vphi(\ee,t)}\sig^f(\xx,t) \nn^f da(\xx) + \int_{\ee} \sig^s(\zz,t) \nn^s da(\zz)=0
      \qquad \forall  \ee\subset\Gam. \label{dn1.4}
\end{align}
Here $\Gam=\pa\Omg^f_{(0)} \cap \pa \Omg^s$ is the FS
interface at $t=0$ and $\ee$ is any part of it. $\vphi(\ee,t)$ is
the image of $\ee$ at time $t$ under the mapping $\vphi$ and it is
part of the fluid boundary. $\nn^f$ and $\nn^s$ are the
outward normals.
Besides the interface conditions \qref{dn1.3} and \qref{dn1.4}, we
also need boundary conditions on the rest of the boundaries. These fixed boundaries are called $\Sigma_1$ to $\Sigma_4$:
$\pa\Omg^f_{(t)}=\vphi(\Gam,t)\cup\Sigma_1\cup\Sigma_2$, $\pa
\Omg^s=\Gam\cup \Sigma_3\cup \Sigma_4$. We require
\begin{equation} \label{dn1.5}
  \uu|_{\Sigma_1}=\uu_b, \qquad \sig^f\nn^f|_{\Sigma_2}=\sig^f_b, \qquad \vphi|_{\Sigma_3}=\vphi_b, \qquad \sig^s \nn^s|_{\Sigma_4} = \sig^s_b.
\end{equation}
Here $\uu_b$, $\sig^f_b$, $\vphi_b$ and $\sig^s_b$ are prescribed fluid velocity,
fluid traction, solid position and solid traction on the $\Sigma_i$'s.



The well-posedness of FS interaction problems has
been studied for various models (see \cite{DE,G,CoS} and the
references in \cite{CoS}). When the solid is deformable, the
analysis uses Lagrangian description for both the fluid
and the solid \cite{G,CoS}. But for computational efficiency, we use
Arbitrary Lagrangian Eulerian (ALE) description for fluid in this paper.
We learn from \cite{FN1} the idea of using conservative ALE description
with proper intermediate mesh or meshes so that the stability
bound obtained does not explicitly depend
on the mesh movement. But \cite{FN1} only considers the convection
diffusion equation on a domain whose motion is known.


We will show that \qref{dn1.4} is equivalent to \qref{antman.4} while the latter is more popular in the literature.
Schemes based on \{\qref{dn1.1},\qref{dn1.2},\qref{dn1.3}, \qref{antman.4},
\qref{dn1.5}\} can be classified into two classes: partitioned schemes and monolithic schemes. In a partitioned scheme
(\cite{CGN,FZG}), different solvers are used for fluid and elastic equations. For stability reason, one may hope
that the two interface conditions \qref{dn1.3} and \qref{antman.4} are
satisfied simultaneously. But this cannot be achieved with one single
iteration between the fluid and the solid solvers.
Various improvements have been proposed:
\cite{FGG,BQQ} reduce the fluid-structure coupling to pressure-structure coupling after a temporal discretization;
\cite{GGCC} proposes a beautiful stable splitting scheme for fluid-membrane interaction problem;
\cite{GGNV} discusses how to achieve fast convergence by properly choosing the boundary conditions for different solvers.
So far, all the numerical analysis requires that the fluid domain does not move and
the flow is Stokes type \cite{FGG,GGCC,GGNV}.
These assumptions conceal the important fact that the convection term
in the Navier-Stokes equations indeed stabilizes the scheme when the domain moves.
In a monolithic scheme (\cite{TH1,BDR}), the
governing equations for the fluid and the solid as well as the governing
equation for the displacement of the mesh are coupled and solved
all together. Hence a large nonlinear system is required to solve
in each time step \cite[remark below (3.11)]{BDR} \cite[section
3.5]{TH1}. Since the interface conditions \qref{dn1.3} and \qref{antman.4}
are built into the system, they are
automatically satisfied once the system is solved. Then it is easy
to {\em believe} that monolithic schemes will preserve the stability of the
associated continuous models. However, as the discrete schemes so far
proposed are very complicated, we are not aware of any {\em proof} of
existence and stability of numerical solutions in the literature.

The celebrated immersed boundary method of Peskin \cite{Pe} uses
delta function to represent the force at the FS interface.
When the solid is codimension-1, i.e., a surface in $\RR^3$ or a curve in $\RR^2$,
there comes the immersed interface method of LeVeque and Li \cite{LL}
which is spatially more accurate. In this method, instead of using delta
function, the forces are translated into the jump conditions
across the interface. These jump conditions are then taken care by changing the
discretization of the differential operators at
stencils acrossing the interface \cite{LL,BS}.
To our point of view, there are still some aspects left to be improved for the methods initiated by Peskin, LeVeque and Li:
Immersed boundary method {\em in general} is only first order in space \cite[p.500,p.509]{Pe}, \cite[p.4]{BS};
Immersed interface method cannot handle solid with a finite volume.

In this paper, we will present a new monolithic scheme. It is based on a
new formulation of FS interaction which in some aspect is similar to the well-known one field formulation of multi-fluid flows
(see \cite[(1)]{UT} for example, but you will not see any delta function in our method).
In our scheme, the unknown is $(\uu;p;\vv)$, being the fluid velocity, fluid pressure and solid velocity.
There are many nice features of our scheme:

\begin{itemize}

 \item[1)] (Explicit mesh moving)
Our scheme is explicit interface advancing which means that the FS interface at time $t^n$ is
constructed explicitly using only information of the solid at time $t^{n-1}$.
So, determining the FS interface and computing $(\uu;p;\vv)$ are decoupled.

 \item[2)] (Smaller system) Let $\uu^f$ be the fluid velocity.
Let $\uu^s$ and $\vphi^s$ be the solid velocity and position.
Let $\uu^{fm}$ and $\vphi^{fm}$ be the fluid mesh velocity and position.
Then, ignoring the pressure for simplicity, \cite[(37)]{TH1}'s unknown is $(\uu^f;\uu^s;\vphi^s;\vphi^{fm})$,
and \cite[(3.10)]{BDR}'s unknown is $(\uu^f;\uu^s;\vphi^s;\uu^{fm};\vphi^{fm})$.
Our unknown is $(\uu^f;\uu^s)$.
In our scheme, the mesh related information and the solid position are all
treated explicitly.

 \item[3)] (Linear system which is solved only once) Our scheme only asks
to solve a linear system once per time step when
it is Navier-Stokes flow and linear elastic structure coupling. But \cite{TH1,BDR}
ask to solve a nonlinear system with unknowns mentioned before per time step.

 \item[4)] (Easy implementation) As our scheme is Jacobian free and does not require characteristics,
it is very easy to implement.

 \item[5)] (Stability) The most important feature of our scheme is its stability:
Even though the fluid mesh is constructed explicitly,
the total energy of the fluid and the solid at time $t^{n}$
is bounded by the total energy at time $t^{n-1}$.
The stability bound hence obtained does not explicitly depend
on the fluid mesh velocity which however is explicitly used in the computation through the ALE
formulation (see Theorem~\ref{T.stab.fluid} which is our main theorem).
This feature is shared by the original continuous model (see Theorem~\ref{P.stab}).
Surprisingly, the nonlinear convection term in the Navier-Stokes equations
plays a crucial role in proving Theorem~\ref{T.stab.fluid} and the stability result does not apply to Stokes flow.

\end{itemize}

However, as the energy bound alone at $t^{n-1}$ is not enough to prevent the
FS interface from colliding with itself or other fixed boundaries at $t^{n}$, we have
to {\em assume} (but see Remark~\ref{R.Remark1}) that the $\Dt$ we take satisfies the following collision free condition
so that we are able to construct
the fluid mesh at time $t^n$:
\begin{equation}\label{stab.assum.1}
\begin{array}{ll}
 & \text{the time step $\Dt$ that we take will not make the FS} \\
 & \text{interface collides with itself or other fixed boundaries}.
\end{array}
\end{equation}

As the interface is determined by extrapolation,
one may wonder what is the point to use a monolithic scheme --- can the
unknowns be easily solved for separately?
If one try that, the scheme becomes a loosely coupled partitioned scheme. Consequently, there will be time lag
in the enforcement of the two continuity conditions. How this time lag affects the stability
will not be addressed in this paper and we refer to \cite{CGN,FZG} and the references therein.
By paying the price of solving a larger system,
monolithic scheme enables the satisfaction both \qref{dn1.3} and
\qref{antman.4} at the same time which is the key to get stability in all density ratio regime.
How to solve the system for $(\uu;p;\vv)$ efficiently will not be addressed in this paper,
even though we do have a {\em linear} system when it is Navier-Stokes and linear elasticity coupling.

As extrapolation is used to determine the interface position,
one may wonder whether it will damage the accuracy.
We note that the position of any point $\zz\in\Gam$ satisfies
$\vphi(\zz,t^n)=\vphi(\zz,t^{n-1})+\int_{t^{n-1}}^{t^n} \vv(\zz,t)dt$ where $\vv$ is the solid velocity.
We know using the right end point rule (hence implicit) to approximate the
integral is not necessarily more accurate than using the left end point rule (hence explicit).
Obviously, left end point rule is cheaper.
Less obviously, as we have mentioned and will prove later, left end
point rule is also stable if it is handled properly.
We will present numerical test that verifies the first order temporal accuracy
in Section~\ref{S.num} (see Table~\ref{Table.accu} and Fig.~\ref{F.interface}). Higher order schemes will be
discussed in a forthcoming paper.


The rest of the paper is organized as follows: In
Section~\ref{S.w.fe} we introduce the combined field formulation and
its weak form for FS interaction, and also discuss the conservative
ALE description of it. We present our scheme in Section 3, starting
from the discretization of the solid part. Then we discuss the fluid
mesh construction, the ALE mapping $\Phi^n$ and its applications. In
particular, we can define the backward in time
(\reflectbox{$\vec{\quad}$}) and the forward in time ($\vec{\quad}$)
extension (see \qref{phi-st} and \qref{vecu}). Before we introduce
our scheme \qref{fluid.1}, we mention two results
(Lemma~\ref{L.gcl.1} and Corollary~\ref{L.int.unun}) which are
related to Geometric Conservation Law and present Lemma~\ref{L.wu}
which contributes to some crucial cancelation that will be used in
the stability proof. After this long preparation, we prove the
stability of our scheme in Section~\ref{S.conservative.stab}. This
is our main theoretical result (Theorem~\ref{T.stab.fluid}). Two
numerical tests will be discussed in Section~\ref{S.num}. They
verfies the stability as well as the first order temporal accuracy
of our scheme.

We emphasize that \qref{dn1.4} does not contain any Jacobian and it fits perfectly into the
framework of (isoparametric) finite element methods.
For \qref{dn1.4}, the author thanks Dr. Stuart Antman for the teaching
of elasticity around year 2002 which motivates most of the work in this paper.


\section{Combined field formulation using ALE description} \label{S.w.fe}
We show that the FS system \{\qref{dn1.1},\qref{dn1.2},\qref{dn1.3},\qref{dn1.4},\qref{dn1.5}\}
has a very clean and simple weak formulation.
Then we put it into ALE format using the conservative formulation.
To save space, we will not discuss the non-conservative formulation but refer the readers to
\cite{Liu2}.

\subsection{Traction boundary condition}
We first show that we can insert test functions into the traction
boundary condition \qref{dn1.4} (\cite[page 489]{An}).
\begin{lemma} \label{L.antman}
Let $\Gam=\pa\Omg^f_{(0)} \cap \pa \Omg^s$ be the FS interface at $t=0$. If \qref{dn1.4} is true, then for any
$\ee\subset \Gam$ and any $\tphi^s$ that is defined on $\Gam$,
\begin{equation} \label{bdc}
  \int_{\vphi(\ee,t)} \tphi^f(\xx)\cdot \(\sig^f(\xx,t) \nn^f\)
  da(\xx) + \int_{\ee} \tphi^s(\zz) \cdot \(\sig^s(\zz,t) \nn^s\) da(\zz)=0,
\end{equation}
where $\tphi^f(\xx)$ is defined on $\vphi(\Gam,t)$ and satisfies
$\tphi^f(\vphi(\zz,t))=\tphi^s(\zz)$.
\end{lemma}

\begin{proof}
Consider $d=3$. Suppose $\ee\subset \Gam$ is the image of the
mapping $\zz=\zz(s_1,s_2)$ with $\ss=(s_1,s_2)\in \vpi$.
Then, when $\xx=\vphi(\zz,t)$, the mapping $\ss\mapsto\zz\mapsto\xx$ maps $\vpi$ to $\ee$ and then to $\vphi(\ee,t)$. So,
we can change all the integrations to $\vpi$:
\begin{equation}
\int_{\vphi(\ee,t)} \!\!\!\!\tphi^f(\xx)\!\cdot \!\(\!\sig^f\!(\xx) \nn^f\)\!
da(\xx) \!=\!  \int_{\vpi} \!\!\tphi^f\(\xx(\ss)\) \!\cdot \!\(\!
\sig^f\!\(\xx(\ss)\) \(\frac{\pa \xx}{\pa s_1} \times \frac{\pa
\xx}{\pa s_2}\!\)\!\) d s_1 d s_2, \label{bdit.1}
\end{equation}
\begin{equation}
\int_{\ee} \tphi^s(\zz)\cdot\(\sig^s(\zz) \nn^s\) da(\zz) =
-\int_{\vpi} \tphi^s(\zz(\ss))\cdot \(\sig^s(\zz(\ss)) \(\frac{\pa
\zz}{\pa s_1}\times\frac{\pa \zz}{\pa s_2}\)\) ds_1 ds_2.
\label{bdit.2}
\end{equation}
With the same idea, \qref{dn1.4} can be rewritten as
\begin{equation} \label{antman.2}
 \int_{\vpi} \sig^f\(\xx(\ss)\)
 \left(\frac{\pa \xx}{\pa s_1} \times \frac{\pa \xx}{\pa s_2}\right)
 -
 \sig^s(\zz(\ss))
 \left(\frac{\pa \zz}{\pa s_1}\times\frac{\pa \zz}{\pa s_2}\right) \; ds_1ds_2=0.
\end{equation}
As the above equation is true for any $\vpi$, the integrand must be zero.
Then, since $\tphi^f\(\xx(\ss)\)=\tphi^s\(\zz(\ss)\)$, we get
\begin{equation} \label{bdit.3}
 \tphi^f\(\xx(\ss)\) \cdot\( \sig^f\(\xx(\ss)\)
 \left(\frac{\pa \xx}{\pa s_1} \times \frac{\pa \xx}{\pa s_2}\right)\)
 -
 \tphi^s\(\zz(\ss)\) \cdot \( \sig^s(\zz(\ss))
 \left(\frac{\pa \zz}{\pa s_1}\times\frac{\pa \zz}{\pa s_2}\right)\)  =0.
\end{equation}
Now, putting \qref{bdit.1},\qref{bdit.2},\qref{bdit.3} together, we obtain \qref{bdc}.
\end{proof}

\medskip

\noindent {\bf Remark}: Since
$
 \frac{\pa \xx}{\pa s_1} \times \frac{\pa \xx}{\pa s_2}
 =\(\frac{\pa \xx}{\pa \zz}\frac{\pa \zz}{\pa s_1}\) \times \(\frac{\pa \xx}{\pa \zz}\frac{\pa \zz}{\pa s_2}\)
 =\(\det\frac{\pa \xx}{\pa \zz}\)\(\frac{\pa \xx}{\pa \zz}\)^{-\top} \(\frac{\pa \zz}{\pa s_1} \times \frac{\pa \zz}{\pa s_2}\)
$ and $\left(\frac{\pa \zz}{\pa s_1} \times \frac{\pa \zz}{\pa
s_2}\right)ds_1ds_2 = -\nn^s da(\zz)$, \qref{antman.2} can be
written as
\begin{equation} \label{antman.3}
 \int_{\vpi} \sig^f\(\xx\) \(\det\frac{\pa \xx}{\pa \zz}\) \(\frac{\pa \xx}{\pa \zz}\)^{-\top} \nn^s    -
 \sig^s(\zz) \nn^s \; da(\zz)=0.
\end{equation}
So, we can use
\begin{equation} \label{antman.4}
 \(\det\frac{\pa \xx}{\pa \zz}\)\sig^f\(\xx(\zz)\) \(\frac{\pa \xx}{\pa \zz}\)^{-\top} \nn^s    -
 \sig^s(\zz) \nn^s =0 \quad \mbox{ on }\Gam
\end{equation}
as the boundary condition for the continuity of traction (e.g.
\cite[page 489, (15.34)]{An}). However, as \qref{antman.4} contains
Jacobian, it is not as friendly to numerics as \qref{bdc}.

\subsection{Combined field formulation}
We will use velocity field as the unknown for both fluid and solid because
it leads to a simple enforcement of
the boundary condition \qref{dn1.3}.
The combined field formulation for time dependent FS interaction is
as follows:
\begin{align}
     \rho^f (\pa_t \uu+\uu\cdot\grad\uu) = \div \sig^f(\uu,p) + \rho^f\gg^f, \qquad \div\uu=0 \qquad\mbox{ in }\Omg^f_{(t)},    \label{dn6.1}\\
     \rho^s \pa_t \vv = \div \sig^s(\vphi_0+\mbox{$\int_{0}^t$} \vv(\tau)d\tau) +\rho^s\gg^s \qquad\mbox{ in }\Omg^s, \label{dn6.2} \\
     \uu(\vphi(\zz,t),t)=\vv(\zz,t) \qquad \forall \zz\in \Gam, \label{dn6.3} \\
     \int_{\Gam} \tphi^s(\zz)\sig^s(\zz,t) \nn^s da(\zz)+\int_{\vphi(\Gam,t)} \!\!\!\tphi^f(\xx)\sig^f(\xx,t) \nn^f
     da(\xx)=0  \; \quad \forall \tphi^s, \label{dn6.4} \\
     \uu|_{\Sigma_1}=\uu_b, \qquad \sig^f\nn^f|_{\Sigma_2}=\sig^f_b, \qquad
     \vv|_{\Sigma_3}=\pa_t\vphi_b, \qquad \sig^s\nn^s|_{\Sigma_4}=\sig^s_b, \label{dn6.5}
\end{align}
where $\tphi^f(\xx)$ is defined by $\tphi^s$ through
$\tphi^f(\vphi(\zz,t))=\tphi^s(\zz)$ for any $\zz\in \Gam$.
$\Gam$ is the FS interface at $t=0$.
$\vphi(\zz,t)=\vphi_0(\zz)+\mbox{$\int_{0}^t$} \vv(\zz,\tau)d\tau$.
$\Sigma_1\cup \Sigma_2=\pa\Omg^f_{\t}\backslash \vphi(\Gam,t)$ and
$\Sigma_3\cup \Sigma_4=\pa\Omg^s\backslash \Gam$ are fixed boundaries. At the initial time $t=0$, we are given $\uu_0$, $\vv_0$ and $\vphi_0$
being the initial fluid velocity, initial solid velocity and initial solid position.

Now we will derive a weak form of \qref{dn6.1}--\qref{dn6.5}. Suppose we know a function $\vphi$ defined on $\Gam$,
introduce
\begin{align}
 V^{\vphi}_{(t)} = \Big\{(\uu;p;\vv): \quad & \uu\in H^1(\Omg^f_{(t)}), \quad p\in L^2(\Omg^f_{(t)}), \quad
 \vv\in W^{1,\infty}(\Omg^s), \nonumber \\
 & \uu(\vphi(\zz))=\vv(\zz) \; \forall \zz\in \Gam \Big\}.
\end{align}
Note that $\uu$ and $\vv$ are defined on different domains (see Fig.~\ref{F.1} for an illustration).
$\Omg^s$ is given from the very beginning and will never change. $\Omg^f_{(t)}$ is varying with respect to
$t$. We require $\vv\in W^{1,\infty}(\Omg^s)$ since $\sig^s$ can be rather nonlinear.
But for linear elasticity \qref{lin.material}, requiring $\vv\in H^1(\Omg^s)$ is enough. 

\begin{theorem} \label{T.weak}
The system \qref{dn6.1}--\qref{dn6.5}
has the following weak form: We are looking for $(\uu(\cdot,t);p(\cdot,t);\vv(\cdot,t))\in
V^{\vphi(\cdot,t)}_{(t)}$ with $\uu|_{\Sigma_1}=\uu_b$,
$\vv|_{\Sigma_3}=\pa_t\vphi_b$, such that for any $t\in[0,T]$,
\begin{align}
     \<\rho^f \(\pa_t \uu+\uu\cdot\grad\uu\), \tphi^f\>_{\Omg^f_{(t)}} + \<\sig^f(\uu,p), \grad \tphi^f\>_{\Omg^f_{(t)}}
     - \<\div\uu, q^f \>_{\Omg^f_{(t)}}\nonumber \\
     +  \<\rho^s \pa_t \vv ,\tphi^s\>_{\Omg^s} + \<\sig^s(\vphi_0+\mbox{$\int_{0}^t$} \vv(\tau)d\tau), \grad\tphi^s\>_{\Omg^s}\nonumber \\
     = \<\rho^f\gg^f,\tphi^f\>_{\Omg^f_{(t)}} + \<\sig_b^f,\tphi^f\>_{\Sigma_2} + \<\rho^s\gg^s,\tphi^s\>_{\Omg^s}
     + \<\sig_b^s,\tphi^s\>_{\Sigma_4} \label{onefield.1}
\end{align}
for any $(\tphi^f;q^f;\tphi^s)\in V^{\vphi(\cdot,t)}_{(t)}$ satisfying
$\tphi^f|_{\Sigma_1}=0$, $\tphi^s|_{\Sigma_3}=0$. The $\vphi(\cdot,t)$ in $V^{\vphi(\cdot,t)}_{(t)}$
is defined by $\vphi(t)=\vphi_0+\mbox{$\int_{0}^t$} \vv(\tau)d\tau$.
\end{theorem}
\begin{proof}
The boundary condition \qref{dn6.3} has been built into the function space $V^{\vphi(\cdot,t)}_{(t)}$.
We dot (\ref{dn6.1}a) with $\tphi^f$, dot (\ref{dn6.1}b) with $q^f$ and dot \qref{dn6.2} with $\tphi^s$.
After integration by parts for the $\sig^f$ and $\sig^s$ terms, we add the resulting
equations together. The boundary integrals on $\vphi(\Gam,t)$
and $\Gam$ cancel each other because of \qref{dn6.4} and the definition of $V^{\vphi(\cdot,t)}_{(t)}$.
\end{proof}


%

\subsection{Stability identity for \qref{onefield.1}}
We want to derive some stability identity for \qref{onefield.1}
when the solid has a strain energy \qref{strain}.
Once again, we denote $\grad\vphi(\zz,t)$ by $\FF(\zz,t)$. With the $I_s$ defined in \qref{strain},
\begin{equation} \label{stab.dtI}
 \frac{d}{dt} I_s(\vphi(\cdot,t)) = \int_{\Omg^s} \frac{\pa W(\FF)}{\pa \FF}:\grad \pa_t\vphi(\zz,t) d\zz
 = \<\sig^s(\vphi(\cdot,t)),\grad\vv(\cdot,t)\>_{\Omg^s}.
\end{equation}

We have the following identity for divergence free velocity field
which can be applied to the first term in \qref{onefield.1} when
$\uu_b=0$ and $\Sigma_2=\emptyset$:
\begin{lemma} \label{L.stab}
Consider a divergence free velocity field $\uu(\xx,t)$ defined on a
time varying domain $\Omg^f_{(t)}$. Assume at any point on
$\pa\Omg^f_{(t)}$, either $\uu\cdot\nn=\ww\cdot\nn$ or $\uu=0$ where
$\nn$ is the outward normal and $\ww$ is the velocity that
$\pa\Omg^f_{(t)}$ moves. Then
\begin{equation} \label{Reynolds.1}
 \frac12 \frac{d}{dt} \|\uu\|^2_{\Omg^f_{(t)}} = \< \pa_t \uu+\uu\cdot\grad\uu, \uu\>_{\Omg^f_{(t)}} .
\end{equation}
\end{lemma}
\begin{proof}
Recall the following Reynolds transport theorem
(\cite[page 488, (15.23)]{An}): for any $\eta$,
\begin{equation} \label{Reynolds.2}
 \frac{d}{dt}\int_{\Omg_{\t}^f} \eta(\xx,t) d\xx = \int_{\Omg_{\t}^f} \pa_t \eta(\xx,t) + \int_{\pa \Omg_{\t}^f} \eta(\xx,t)\ww(\xx,t)\cdot\nn.
\end{equation}
So,
\begin{align}
 \frac12 \frac{d}{dt} \|\uu\|^2_{\Omg^f_{(t)}} = & \frac12 \frac{d}{dt}\int_{\Omg^f_{\t}} |\uu(\xx,t)|^2 d\xx
 = \int_{\Omg^f_{\t}} \pa_t \uu \cdot\uu + \frac12 \int_{\pa \Omg^f_{\t}}  |\uu|^2 \ww\cdot\nn \nonumber \\
 =& \int_{\Omg^f_{\t}} \pa_t \uu \cdot\uu + \frac12 \int_{\pa \Omg^f_{\t}}  |\uu|^2 \uu\cdot\nn. \label{temp.1}
\end{align}
In the last step we have used the condition either
$\uu\cdot\nn=\ww\cdot\nn$ or $\uu=0$ on $\pa\Omg^f_{\t}$. Using the divergence free condition, the last surface integral
can be rewritten as
\begin{align}
 \frac12 \int_{\pa \Omg^f_{\t}}  |\uu|^2 \uu\cdot\nn
 =\int_{\Omg^f_{\t}} \frac12 \div\(|\uu|^2 \uu\) d\xx
 =\int_{\Omg^f_{\t}} (\uu\cdot\grad\uu) \cdot\uu d\xx. \label{temp.2}
\end{align}
We have used $\div\uu=0$ in the last step. Plugging \qref{temp.2} into \qref{temp.1}, we get \qref{Reynolds.1}.
\end{proof}

Combining \qref{stab.dtI} and \qref{Reynolds.1}, if we let $\tphi^f=\uu$, $q^f=-p$
and $\tphi^s=\vv$ in \qref{onefield.1},
we obtain the following result:
\begin{theorem} \label{P.stab}
When the solid has strain energy \qref{strain} and when $\uu_b=0$,
$\Sigma_2=\emptyset$, $\pa_t\vphi_b=0$ in \qref{dn6.5},
the solution of \qref{onefield.1} satisfies
\begin{align}
 \frac{\rho^f}{2}&\frac{d}{dt}\|\uu(\cdot,t)\|_{\Omg^f_{(t)}}^2
  + \frac{\rho^f\nu^f}{2}\|\grad\uu+\grad\uu^\top\|_{\Omg^f_{(t)}}^2
  + \frac{\rho^s}{2}\frac{d}{dt}\|\vv(\cdot,t)\|_{\Omg^s}^2
  \nonumber \\
 &+ \frac{d}{dt} I_s(\vphi_0\!+\!\mbox{$\int_{0}^t$}\vv(\tau)d\tau)
= \<\rho^f\gg^f,\uu\>_{\Omg^f_{(t)}} + \<\rho^s\gg^s,\vv\>_{\Omg^s} + \<\sig_b^s,\vv\>_{\Sigma_4}. \label{stab.f.3}
\end{align}
\end{theorem}

\subsection{Conservative formulation of Arbitrary Lagrangian Eulerian (ALE) description}

When the time reaches $t^n$ (which can be any number), we choose the $\Omg^f_{(t^n)}$ as the reference domain
and construct a backward in time mapping $\Phi^n(\cdot,t)$ that maps $\xx\in \Omg^f_{(t^n)}$ to $\yy=\Phi^n(\xx,t)\in\Omg^f_{(t)}$
for any $t \le t^n$. See \qref{Phit} for a way of constructing $\Phi^n$. Note that $\Phi^n(\xx,t^n)=\xx$.
Define
\begin{equation}
 \ww^n(\xx,t)=\pa_t\Phi^n(\xx,t). \label{ww}
\end{equation}

With this $t^n$ fixed, we set $t=t^n$ in \qref{onefield.1} and choose test function $\tphi^f(\cdot,t)$ satisfing
\begin{equation} \label{ale.22}
  \frac{d}{dt}\tphi^f(\Phi^n(\xx,t),t)=0  \qquad \forall x\in \Omg_{(t^n)}.
\end{equation}
In a finite element method, requiring \qref{ale.22} for a basis function $\tphi^f$ means
the function is always attached to the nodal point it starts from no matter how the mesh moves.
See \qref{phi-st},\qref{invariant}.
Evaluating \qref{ale.22} at $t=t^n$ and using $\Phi^n(\xx,t^n)=\xx$, we obtain
\begin{equation} \label{ale.41}
 \pa_t \tphi^f(\xx,t^n)=-(\pa_t\Phi^n(\xx,t^n))\cdot\grad\tphi^f(\Phi^n(\xx,t^n),t^n) = -\ww^n(\xx,t^n)\cdot\grad\tphi^f(\xx,t^n).
\end{equation}

As $\ww^n$ is exactly the velocity of domain $\Omg^f_{(t^n)}$, we can apply the Reynolds transport theorem \qref{Reynolds.2}
with $\eta = \uu(\xx,t) \cdot \tphi^f(\xx,t)$ to get
\begin{align}
  & \left.\frac{d}{dt}\right|_{t=t^n}\!\!\<\uu(\cdot,t),\tphi^f(\cdot,t)\>_{\Omg^f_{\t}}
 \!\!=\!\int_{\Omg^f_{(t^n)}} \!\! (\pa_t \uu)\cdot\tphi^f + \uu \cdot (\pa_t \tphi^f) + \div\(\uu\cdot\tphi^f \ww^n\), \label{ale.23}
\end{align}
where the integrand on the right hand side is evaluated at $t=t^n$. Because of \qref{ale.41},
the sum of the last two terms in the integrand equals
$\uu \cdot (\pa_t \tphi^f) + \div\(\uu\cdot\tphi^f \ww^n\) = \tphi^f_k \pa_j \(\uu_k \ww^n_{j}\) =\(\div(\uu\otimes \ww^n)\)\cdot\tphi^f$.
So, using $\uu\cdot\grad\uu=\div\(\uu\otimes\uu\)$, \qref{ale.23} leads to
\begin{align}
 & \<\pa_t \uu(\cdot,t^n)+\uu(\cdot,t^n)\cdot\grad\uu(\cdot,t^n), \tphi^f(\cdot,t^n)\>_{\Omg^f_{(t^n)}} \nonumber \\
  = & \left.\frac{d}{dt}\right|_{t=t^n}\<\uu(\cdot,t),\tphi^f(\cdot,t)\>_{\Omg^f_{(t)}} \!\!
    +\<\div\(\uu\otimes(\uu-\ww^n(\cdot,t^n))\), \tphi^f(\cdot,t^n) \>_{\Omg^f_{(t^n)}}\label{ale.24}
\end{align}
To summarize, if
$(\uu(\cdot,t);p(\cdot,t);\vv(\cdot,t))\in
V^{\vphi(\cdot,t)}_{(t)}$ with $\uu|_{\Sigma_1}=\uu_b$,
$\vv|_{\Sigma_3}=\pa_t\vphi_b$ is a solution to \qref{dn6.1}--\qref{dn6.5},
then $(\uu;p;\vv)$ satisfies the following:
For any $t^n\le T$,
for any given backward in time mapping $\Phi^n(\cdot,t):\Omg_{(t^n)}^f\to \Omg_{(t)}^f$,
for any $\tphi^f(\cdot,\cdot)$ defined in the space-time domain of the fluid, and satisfying
$\left.\frac{d}{dt}\right|_{t=t^n}\tphi^f(\Phi^n(\cdot,t),t)=0$ and $\tphi^f(\cdot,t^n)|_{\Sigma_1}=0$,
for any $\tphi^s$ defined in $\Omg^s$ and satisfying
$\tphi^s|_{\Sigma_3}=0$ and $(\tphi^f(\cdot,t^n);q^f;\tphi^s)\in V^{\vphi(\cdot,t^n)}_{(t^n)}$,
\begin{align}
     &\left.\frac{d}{dt}\right|_{t=t^n}\!\!\<\uu(\cdot,t),\tphi^f(\cdot,t)\>_{\Omg^f_{(t)}} \!\!\!
     +\<\nabla\!\cdot\!\(\uu\otimes(\uu\!-\!\ww^n(\cdot,t^n))\), \tphi^f(\cdot,t^n) \>_{\Omg^f_{(t^n)}} \nonumber \\
     & + \<\sig^f(\uu,p), \grad \tphi^f(\cdot,t^n)\>_{\Omg^f_{(t^n)}}- \<\div\uu, q^f \>_{\Omg^f_{(t^n)}} \nonumber \\
     & +\<\rho^s \pa_t \vv(\cdot,t^n) ,\tphi^s\>_{\Omg^s}
     + \<\sig^s(\vphi_0+\mbox{$\int_{0}^{t^n}$} \vv(\tau)d\tau), \grad\tphi^s\>_{\Omg^s} \nonumber \\
     = & \<\rho^f\gg^f,\tphi^f\>_{\Omg^f_{(t^n)}} + \<\sig_b^f,\tphi^f\>_{\Sigma_2} + \<\rho^s\gg^s,\tphi^s\>_{\Omg^s}
     + \<\sig_b^s,\tphi^s\>_{\Sigma_4}. \label{onefield.conserve}
\end{align}

\section{Fully discrete scheme}
Now we turn to the fully discrete scheme.
The initial set up is as follows:
First of all, there is a mesh $\TT^s_h$ for the
solid reference domain $\Omg^s$. Part of the boundary grid points of $\TT^s_h$ form a mesh
$\TT^\Gam_h$ for $\Gam$.
Suppose we are given $\uu_h^{n-1}$, $\vv_h^{n-1}$, $\vphi_h^{n-1}$ and
suppose we also have the mesh $\TT^f_{h,t^{n-1}}$ of domain $\Omg^f_{h,t^{n-1}}$ where $\Omg^f_{h,t^{n-1}}$
is our numerical approximation for the fluid domain at time $t^{n-1}$.
Without loss of generality, we use $\PP_m/\PP_{m-1}/\PP_m$
elements for fluid velocity, fluid pressure and solid velocity.
For efficiency and also to meet the requirement of optimal
isoparametric finite element mesh, we require that an edge of
$\TT^f_{h,t^{n-1}}$ (or $\TT^s_h$) is straight when it does not belong to the
boundary of $\Omg^f_{h,t^{n-1}}$ (or $\TT^s_h$)
and is curved otherwise. Here edge
refers to both surface and edge if $d=3$.

\subsection{Solid part}
The first order in time discretization for the solid part (3rd line in \qref{onefield.conserve}) is rather simple.
So we discuss it first.
Recall at time $t^{n-1}$, we have $\vv_h^{n-1}$
and $\vphi_h^{n-1}$ on $\Omg^s$. For the next moment $t^n$, define
\begin{equation} \label{vphihn}
  \vphi_h^n = \vphi_h^{n-1}+\Dt\vv_h^{n-1}.
\end{equation}
Through out this paper, $\vphi_h^n$ is always constructed explicitly in this way.
To determine $\vv_h^n$, we have two different approaches:

\subsubsection{Efficient semi-implicit discretization for nonlinear material} 
When $\sig^s=\sig^s(\vphi)$, we have:
\begin{equation} \label{increment}
 \<\sig^{s}(\vphi(t^{n+1}\!)),\grad\tphi^s \>_{\Omg^s} \!\approx
 \<\sig^{s}(\vphi(t^{n}\!)),\grad\tphi^s \>_{\Omg^s}\! +
 A_s\!\(\grad\vphi(t^{n}\!);\grad\(\vphi(t^{n+1}\!)\!\!-\!\!\vphi(t^{n}\!)\),\grad\tphi^s\).
\end{equation}
where $A_s\(\grad\vphi;\grad\vpsi,\grad\tphi^s\)
=\frac{d}{d\epsilon}|_{\epsilon=0} \<\sig^{s}(\vphi+\epsilon\vpsi),\grad\tphi^s \>_{\Omg^s}$.
The exact formulas for the variational derivatives of various materials
should be widely available in the literature as they are used to derive the Newton's method for
solving $\<\sig^{s}(\vphi),\grad\tphi^s
\>_{\Omg^s}=\<\rho^s\gg^s,\tphi^s\>_{\Omg^s}$.
For Saint Venant-Kirchhoff material
\qref{svk.constitutive}, 
\begin{align}
 A_s(\FF;\GG,\HH)=\int_{\Omg^s}\Big(\lam^s (\mbox{tr} \EE) \HH:\GG + \frac{\lam^s}{4}
 \mbox{tr}(\HH^\top \FF + \FF^\top \HH)\mbox{tr}(\GG^\top \FF + \FF^\top \GG )
 \nonumber \\
 + \mu^s \EE:(\HH^\top \GG + \GG^\top \HH) + \frac{\mu^s}{2} (\HH^\top \FF + \FF^\top \HH):(\GG^\top \FF+\FF^\top \GG )\Big),
 \label{svk.As}
\end{align}
where $\EE=\frac{1}{2}\(\FF\FF^\top-I\)$ and $\HH:\GG=\mbox{tr}(\HH^\top\GG)$.
$A_s(\FF;\GG,\HH)$ is a bilinear functional of $\GG$ and $\HH$.
Of course, if we use linear constitutive equation \qref{lin.material},
the approximately equal sign in \qref{increment}
becomes the equal sign.

Because of \qref{increment} and \qref{vphihn}, at time $t^n$, the solid part in \qref{onefield.conserve}
(the 3rd line in \qref{onefield.conserve})
can be approximated by
\begin{align}
     \<\rho^s \frac{\vv_h^{n}-\vv_h^{\npp}}{\Dt}\! ,\tphi^s\>_{\Omg^s}
     & + \<\sig^s(\vphi_h^{n}), \grad\tphi^s\>_{\Omg^s}
     + \Dt A_s\(\grad\vphi_h^{n};\grad\vv_h^{n},\grad\tphi^s\) =\cdots, \label{solid.lin.3}
\end{align}
where we use $\cdots$ to denote the fluid part which will be discussed in Section~\ref{S.fluid}.
The above scheme is linear for the unknown $\vv_h^n$.

\subsubsection{Stable implicit discretization for material with convex strain energy} \label{S.solid.stab}
The above semi-implicit discretization is what we will use in our numerical test. But to prove
unconditional stability for more general nonlinear elastic solid,
we need to consider fully implicit discretization. The existence of the solution will be address in
a later section (Section~\ref{S.existence}) after we also include the fluid variables.

Assume the $W(\FF)$ in \qref{strain} is convex, then
$W(\grad (\vphi+\tphi)) \ge W(\grad \vphi) + \frac{\pa W(\grad\vphi)}{\pa \FF}:\grad\tphi$. Therefore, with
$\vphi=\vphi^{n+1}$ and $\tphi=\vphi^n-\vphi^{n+1}$,
\begin{equation} \label{convex}
 I_s(\vphi^{n})\ge I_s(\vphi^{n+1}) + \<\sig^s(\vphi^{n+1}),\grad(\vphi^n-\vphi^{n+1})\>_{\Omg^s}.
\end{equation}
So, to gain stability, we can discretize the solid part in \qref{onefield.conserve} (the 3rd line in \qref{onefield.conserve}) by
\begin{align}
     \<\rho^s \frac{\vv_h^{n}-\vv_h^{\npp}}{\Dt}\! ,\tphi^s\>_{\Omg^s}
     & + \<\sig^s(\vphi_h^{n}+\Dt\vv_h^n), \grad\tphi^s\>_{\Omg^s} =\cdots. \label{solid.lin.6}
\end{align}
Because of \qref{vphihn} and \qref{convex},
\[
 \Dt\<\sig^s(\vphi_h^{n}+\Dt\vv_h^n), \grad\vv_h^n\>_{\Omg^s}\!=\!\<\sig^s(\vphi_h^{n+1}),\grad(\vphi_h^{n+1}-\vphi_h^n)\>_{\Omg^s}
 \!\!\ge I_s(\vphi_h^{n+1})-I_s(\vphi_h^{n}).
\]
Define
\begin{equation} \label{strain.2}
 \IIs(\vv;\vphi)=\frac{\rho^s}{2}\|\vv\|^2_{\Omg^s} + I_s(\vphi)=
 \frac{\rho^s}{2}\|\vv\|^2_{\Omg^s} + \int_{\Omg^s}W(\grad\vphi).
\end{equation}
Letting $\tphi^s=\vv_h^n$, \qref{solid.lin.6}
leads to
\begin{equation} \label{solid.5}
 \frac{1}{\Dt}\IIs(\vv_h^n;\vphi_h^{n+1})-\frac{1}{\Dt}\IIs(\vv_h^{n-1};\vphi_h^{n}) \le \cdots.
\end{equation}
We have used $(a_n-a_{n-1})a_n = \frac12 (a_n^2 - a_{n-1}^2 + (a_n-a_{n-1})^2) \ge \frac12 a_n^2 - \frac12 a_{n-1}^2$.

\subsubsection{Special case} \label{S.linear}
When the solid is linear elastic (\qref{lin.material}),
discretizations \qref{solid.lin.3} and \qref{solid.lin.6} coincide and their
left hand sides become
\begin{align}
  \<\!\rho^s \frac{\eeta_h^{n\!+\!1}\!-\!2 \eeta_h^{n}\!+\!\eeta_h^{\npp}}{\Dt^2} ,\tphi^s\!\>_{\Omg^s}
  \!\!\!+ \<\!\mu^s\(\grad\eeta_h^{n+1}\!+\!\grad \eeta_h^{n+1,\top}\)\!+\!\lam^s\div\eeta_h^{n+1}\II,\grad\tphi^s\!\>_{\Omg^s}
  \label{solid.lin.4}
\end{align}
with $\eeta_h^n(\zz) = \vphi_h^n(\zz) -\zz$.
So, it is the familiar first order backward differentiation for linear elastodynamics.


\subsection{Fluid part} \label{S.fluid}

Now, let us move to the fluid part in \qref{onefield.1}.
One nice feature of our scheme is that we use explicit interface advancing.
This makes our scheme very efficient, provided that it will not damage
the accuracy (Table~\ref{Table.accu} and Fig~\ref{F.interface})
and the stability (Theorem~\ref{T.stab.fluid}).

So, we define $\vphi_h^n$ by \qref{vphihn} which is an explicit extrapolation.
From the values of $\vphi_h^{n}$ at the grid points of $\TT^\Gam_h$,
we determine the numerical FS interface at time $t^n$.
Moveover, from those values and also the information from fixed boundaries $\Sigma_1$ and $\Sigma_2$,
we are able to determine all the
boundary grid points of $\TT^f_{h,t^n}$. In particular,
we know the position of all boundary {\em vertices} of
$\TT^f_{h,t^n}$. Using the latter as boundary value, we use $\PP_1$
(not $\PP_m$!) element to solve a linear {\em anisotropic}
elasticity equation on $\TT^f_{h,0}$. We use $\mu=\lam=1+\frac{\max_i |T_i|-\min_i|T_i|}{|T_j|}$ for each triangle
$T_j\in \TT^f_{h,0}$. The idea of increasing the stiffness of small
elements to prevent them from being distorted comes from \cite{Ma}.
The result gives positions of {\em vertices} of all the triangles of
$\TT^f_{h,t^n}$. We always used $\TT^f_{h,0}$ to construct
$\TT^f_{h,t^n}$ so that we do not have to reconstruct the stiffness matrix in each time step.
Certainly more sophisticated method can be used, but we shall not discuss those alternatives.

Recall that we use $\PP_m/\PP_{m-1}/\PP_m$ elements. To guarantee optimal rate of
approximation on an isoparametric finite element mesh,
all its interior triangles should be straight and are standard Lagrange elements.
But specific placement of interior grid points on curved triangles
is mandatory \cite{Sc,Ci,Le}.
So, once we have all the vertices and boundary grid points of $\TT^f_{h,t^n}$,
we can determine all the grid points except $G_{\text{ic}}$ which
denotes grid points lying {\em inside} a {\em curved} triangle touching the boundary.
When $m=2$, we are done as $G_{\text{ic}} = \emptyset$.
When $m\ge 3$, we need \cite{Sc,Le}.
We use Scott's procedure \cite{Sc} when $d=2$. Consequently, $G_{\text{ic}}$
is uniquely determined by grid points on the boundary edges (due to the mapping
$(\hat{\lam}_1,\hat{\lam}_2,\hat{\lam}_3)\mapsto (1-\hat{\lam}_2,\hat{\lam}_2,0)$
and the prefactor $\frac{\hat{\lam}_1}{1-\hat{\lam}_2}$ in \cite[(14)]{Le}).
When $d\ge 3$, we need to construct some local chart $\phi$ before we can
use \cite[(22)]{Le} to determine $G_{\text{ic}}$.


\subsubsection{The mapping $\Phi^n$ in the ALE description} \label{S.Phi}
To construct $\Phi^n$, the basic idea is to use the fact that any physical triangle $T$,
no matter which time level it is at, is mapped to {\em the same reference triangle} $\hat{T}\subset \RR^{d}$ with vertices
$\{(0,0,...,0)$,$(1,0,...,0)$,...,\\$(0,...,0,1)\}$.
We use $\Psi^{f,n}_{j}: \hat{T}\to T^{f,n}_j$ to denote this mapping
where $T^{f,n}_j$ denotes the $j$th triangle of $\TT^f_{h,t^n}$.
Then $\Phi^n(\cdot,t)$ is defined piecewisely on each triangle
$T^{f,n}_{j}\in \TT^f_{h,t^n}$ as follows:
\begin{equation} \label{Phit}
 \Phi^n(\xx,t)\Big|_{T^{f,n}_{j}}=\frac{t-t^{n-1}}{\Dt} \xx + \frac{t^{n}-t}{\Dt} \Psi^{f,n-1}_j\circ\(\Psi^{f,n}_{j}\)^{-1}(\xx).
\end{equation}
Here $t\in[t^{n-1},t^n]$ and $\Psi^{f,n-1}_{j}\circ(\Psi^{f,n}_{j})^{-1}$ maps $T^{f,n}_{j}$ to $T^{f,n-1}_j$.
Note that $\Phi^n(\xx,t^n)=\xx$ and $\Phi^n(\cdot,t^{n-1})$ maps $\Omg^f_{h,t^n}$ to $\Omg^f_{h,t^{n-1}}$.
The explicit formula of $\Psi^{f,n}_j$ is well-known \cite{Ci}:
\begin{equation} \label{Psi}
  \xx = \Psi^{f,n}_j(\hat{\xx}) = \sum_{p=0}^{L-1} \hat{\phi}_p(\hat{\xx}) \aa^{n}_{i_{(j,p)}}.
\end{equation}
Here $L$ is the number of grid points on each triangle.
$\hat{\phi}_p$ is the scalar finite element basis function on $\hat{T}$.
$p$ is the local index.
$\aa^n_{i}$ are the $i$th grid point of $\TT^f_{h,t^n}$.
$i$ is the global index. $i_{(j,p)}$ is the mapping from
local index $p$ to global index $i$ on the $j$th triangle.
Using these notations, we have the following result
\begin{lemma} \label{L.fe.phi}
The $\Phi^n(\xx,t)$ defined by \qref{Phit} satisfies
\begin{equation} \label{Phit.2}
  \Phi^n(\xx,t)=\sum_{i=1}^G \phi^{f,n}_i(\xx) \(\frac{t-t^{n-1}}{\Dt} \aa_i^{n} + \frac{t^n-t}{\Dt} \aa_i^{n-1}\)
\end{equation}
where
$\phi^{f,n}_i$ is the scalar finite element basis function on $\TT^f_{h,t^n}$
that is associated with the $i$th grid point and
$G$ is the total number of grid points.
Consequently, $\Phi^n(\xx,t)$
and
\begin{equation}
  \ww^{n}(\xx)=\pa_t\Phi^n(\xx,t).  \label{wwnt}
\end{equation}
are in the finite element space on $\TT^f_{h,t^n}$.
\end{lemma}
\begin{proof}
Because of the way how the basis function $\phi^{f,n}_i$ is defined by $\hat{\phi}_p$,
for any $\xx\in T^{f,n}_{j}\in \TT^f_{h,t^n}$, $\xx=\sum_{p=0}^{L-1} \hat{\phi}_p(\hat{\xx})\aa_{i_{(j,p)}}^{n}
=\sum_{p=0}^{L-1} \phi^{f,n}_{i_{(j,p)}}(\xx)\aa_{i_{(j,p)}}^{n}$.
Moreover, for this $\xx$,
$\Psi^{f,n-1}_j\circ\(\Psi^{f,n}_{j}\)^{-1}\!(\xx)\! = \!\Psi^{f,n-1}_j \( \hat{\xx} \) =
\sum_{p=0}^{L-1} \hat{\phi}_p(\hat{\xx})\aa_{i_{(j,p)}}^{n-1}
=\sum_{p=0}^{L-1} \phi^{f,n}_{i_{(j,p)}}(\xx) \aa_{i_{(j,p)}}^{n-1}$.
Plugging the two previous equations into \qref{Phit}, we obtain \qref{Phit.2}.
\end{proof}

\subsubsection{Intermediate fluid meshes} \label{S.interMesh}

In the conservative ALE scheme, we need intermediate fluid domains that lie between $\Omg^f_{h,t^{n-1}}$
and $\Omg^f_{h,t^n}$. Their constructions make use of the $\Phi^n$ defined in \qref{Phit}.
For $t\in[t^{n-1},t^n]$, define
\begin{equation} \label{Omgt}
 \Omg^f_{h,t}:=\Phi^n(\Omg^f_{h,t^n},t).
\end{equation}
Note that $\Phi^n(\cdot,t)$ maps all the grid points of $\TT^f_{h,t^n}$ to $\Omg^f_{h,t}$
which then form a mesh. We called this mesh $\TT^f_{h,t}$. As $\Psi^{f,n}_j$ is reduced to
affine linear mapping when $T^{f,n}_{j}$ is a straight triangle, it is easy to see that all the interior
triangles of $\TT^f_{h,t}$ are straight. 

\subsubsection{Backward in time extension}
Recall that on $\TT^f_{h,t^n}$, we use $\phi^{f,n}_i(\xx)$ to denote
the scalar basis function associated with the $i$th grid point.
Define its backward in time extension
\begin{equation} \label{phi-st}
  \vecphi_i^{f,n}(\yy,t)=\phi_i^{f,n}([\Phi^n(\cdot,t)^{-1}](\yy))
\end{equation}
for $(\yy,t)\in Q_{[t^{n-1},t^n]}=\{(\yy,t),\yy\in\Omg^f_{h,t},t\in[t^{n-1},t^n]\}$.
An immediate consequence is that
$\vecphi_i^{f,n}(\Phi^n(\xx,t),t)=\phi_i^{f,n}(\xx)$ for any $\xx\in \Omg^f_{h,t^n}$.
Hence
\begin{equation} \label{invariant}
  \frac{d}{dt}\vecphi_i^{f,n}(\Phi^n(\xx,t),t)=0.
\end{equation}

\subsubsection{Related properties}
As $\TT^f_{h,t}$ is
a finite element mesh by itself, for any triangle $T^{f,t}_j\in\TT^f_{h,t}$, automatically there is a mapping $\Psi^{f,t}_j$ that maps
$\hat{T}$ to $T^{f,t}_j$. Like \qref{Psi}, this $\Psi^{f,t}_j$ is given by
$\yy = \Psi^{f,t}_j(\hat{\xx})=\sum_{p=0}^{L-1} \hat{\phi}_p(\hat{\xx}) \aa^t_{i_{(j,p)}}$
where $\aa^t_{i_{(j,p)}}$ is the $p$th grid point of $T^{f,t}_j$ and its global index is $i_{(j,p)}$. Because of \qref{Phit.2}
and the way we construct $\TT^f_{h,t}$, we know
$\aa^t_{i_{(j,p)}} =\Phi^n(\aa^{n}_{i_{(j,p)}},t)= \frac{t-t^{n-1}}{\Dt} \aa^{n}_{i_{(j,p)}} + \frac{t^n-t}{\Dt} \aa^{n-1}_{i_{(j,p)}}$.
On the other hand, the mapping $\Phi^n(\cdot,t)\circ\Psi^{f,n}_{j}$ also maps $\hat{T}$ to $T^{f,t}_j$.  By \qref{Phit}, we have
\begin{align}
  \yy&=[\Phi^n(\cdot,t)\circ\Psi^{f,n}_{j}](\hat{\xx})=\frac{t-t^{n-1}}{\Dt} \Psi^{f,n}_{j}(\hat{\xx})
           + \frac{t^{n}-t}{\Dt} \Psi^{f,n-1}_j(\hat{\xx}) \nonumber \\
  & = \sum_{p=0}^{L-1} \hat{\phi}_p(\hat{\xx}) \( \frac{t-t^{n-1}}{\Dt} \aa^{n}_{i_{(j,p)}} + \frac{t^n-t}{\Dt} \aa^{n-1}_{i_{(j,p)}} \)
    = \sum_{p=0}^{L-1} \hat{\phi}_p(\hat{\xx}) \aa^t_{i_{(j,p)}}. \label{Pit.2}
\end{align}
So we see that $\Phi^n(\cdot,t)\circ\Psi^{f,n}_{j}$ is exactly
$\Psi^{f,t}_j$.
As a result of \qref{Pit.2}, we get $[\Phi^n(\cdot,t)^{-1}](\yy)
=\Psi^{f,n}_{j}(\hat{\xx})$
if $\yy=\Psi^{f,t}_j(\hat{\xx})$.
Consequently, we have the following nice property for numerical implementation.
\begin{proposition} \label{P.basis.extend}
The backward in time extension $\vecphi_i^{f,n}(\yy,t)$ which is defined by \qref{phi-st} is nothing but the standard basis
function associated with the $i$th grid point on $\TT^f_{h,t}$ for any $t\in[t^{n-1},t^n]$.
\end{proposition}
\begin{proof}
Because $[\Phi^n(\cdot,t)^{-1}](\yy)=\Psi^{f,n}_{j}(\hat{\xx})$ for any $\yy\in \Omg_{h,t}^f$ satisfying $\yy=\Psi^{f,t}_j(\hat{\xx})$,
\begin{equation} \label{Pit.3}
  \vecphi_i^{f,n}(\yy,t)=\phi_i^{f,n}([\Phi^n(\cdot,t)^{-1}](\yy))
 =\phi_i^{f,n}(\Psi^{f,n}_{j}(\hat{\xx}))=\hat{\phi}_p(\hat{\xx})
\end{equation}
for some basis function $\hat{\phi}_p$ on $\hat{T}$. The first equality in \qref{Pit.3}
is by the definition \qref{phi-st}.
The last equality in \qref{Pit.3} is by the definition of $\phi_i^{f,n}$. Then as the
$\hat{\xx}$ and $\yy$ in \qref{Pit.3} are linked by $\Psi^{f,t}_j$, \qref{Pit.3} says
$\vecphi_i^{f,n}(\cdot,t)$ is the standard basis function on $\TT^f_{h,t}$.
\end{proof}

If $\ff$ is in the finite element space on $\TT^f_{h,t^n}$,
it has an expansion $\ff=\sum_{i=1}^G \ff_i \phi^{f,n}_i$ with $G$ being the number of grid points on $\TT^f_{h,t^n}$.
Here, $\ff_i=(f_{i,1},...,f_{i,d})$ and the vector scalar product $\ff_i \phi^{f,n}_i=(f_{i,1}\phi^{f,n}_i,...,f_{i,d}\phi^{f,n}_i)$.
We can define the backward in time extension of $\ff$ as follows:
\begin{equation} \label{ff-st}
  \vecf(\yy,t)=\ff([\Phi^n(\cdot,t)^{-1}](\yy))=\sum_{i=1}^G \ff_i \vecphi^{f,n}_i(\yy,t).
\end{equation}


%
%

In particular, recalling that the $\ww^{n}(\xx)$ defined in \qref{wwnt} is
in the finite element space on $\Omg^f_{h,t^n}$,
we have $\ww^{n}(\xx,t)=\sum_{i=1}^{G} \ww^{n}_i \phi^{f,n}_i$.
Its backward in time extension $\vecw^{n}$ in the space-time domain $Q_{[t^{n-1},t^n]}$ is defined as
\begin{equation} \label{wn-st}
  \vecw^{n}(\yy,t)=\ww^{n}([\Phi^n(\cdot,t)^{-1}](\yy))=\sum_{i=1}^{G} \ww^{n}_i \vecphi^{f,n}_i(\yy,t).
\end{equation}

With $\vecw^{n}$ ready, we can study the relation between
$\int_{\Omg^f_{h,t^{n-k}}}\!\!\vecphi_i^{f,n}(\yy,t^{n-k})d\yy$ for $k=0$ and $k=1$.
This is the key to understand the Geometric Conservation Law.
\begin{lemma} \label{L.gcl.1}
Let $\phi^{f,n}_i$ be any function on $\Omg^f_{h,t^n}$ and let $\vecphi_i^{f,n}$
be its backward in time extension defined by \qref{phi-st}. When $\Omg^f_{h,t^n} \subset \RR^2$,
\begin{equation} \label{gcl.4}
  \int_{\Omg^f_{h,t^n}}\!\!\!\!\phi_i^{f,n}(\yy)d\yy
 \!-\!
 \int_{\Omg^f_{h,t^{\npp}}}\!\!\!\!\vecphi_i^{f,n}(\yy,t^{n-1})d\yy = \Dt
 \int_{\Omg^f_{h,t^{\np}}} \!\!\!\!\vecphi_i^{f,n}(\yy,t^{\np}) \div \vecw^{n}(\yy,t^{\np}) d\yy,
\end{equation}
where the divergence in $\div \vecw^{n}$ is taken with respect to the $\yy$ variable.
Note that the first integrand $\phi_i^{f,n}(\yy)$ equals to $\vecphi_i^{f,n}(\yy,t^n)$.
If $\Omg^f_{h,t^n} \subset \RR^3$
the right hand side of \qref{gcl.4} should be
\begin{equation} \label{gcl.gauss}
 \frac{\Dt}{2}\sum_{\ell=1}^2 \int_{\Omg^f_{h,t^{n_\ell}}} \vecphi_i^{f,n}(\yy,t^{n_\ell}) \div \vecw^{n}(\yy,t^{n_\ell}) d\yy
\end{equation}
with $t^{n_1}=t^n-(\frac12+\frac{1}{2\sqrt{3}})\Dt$ and $t^{n_2}=t^n-(\frac12-\frac{1}{2\sqrt{3}})\Dt$
being the two quadrature points of the two-point Gauss quadrature on $[t^{n-1},t^n]$.
\end{lemma}
\begin{proof}
The following argument is essentially the proof of Reynolds transport theorem \qref{Reynolds.2} (see \cite[p.487]{An}).
To simplify the notation, we write $\vecphi^{f,n}_i$ as $\vecphi$.
\begin{align*}
 & \int_{\Omg^f_{h,t^n}}\!\!\vecphi(\yy,t^n)d\yy-
 \int_{\Omg^f_{h,t^{n\!-\!1}}}\!\!\vecphi(\yy,t^{n-1})d\yy
=\int_{t^{n-1}}^{t^n}\( \frac{d}{dt} \int_{\Omg^f_{h,t}} \vecphi(\yy,t)d\yy \)dt \\
= & \int_{t^{n-1}}^{t^n}\( \frac{d}{dt} \int_{\Omg^f_{h,t^n}} \vecphi(\Phi^n(\xx,t),t) \left|\frac{\pa \Phi^n(\xx,t)}{\pa \xx}\right| d\xx \)dt \\
= & \int_{t^{n-1}}^{t^n}\( \int_{\Omg^f_{h,t^n}} \vecphi(\Phi^n(\xx,t),t) \frac{d}{dt} \left|\frac{\pa \Phi^n(\xx,t)}{\pa \xx}\right| d\xx \)dt.
\end{align*}
In the second step, we have changed variables and in the last step we have used \qref{invariant}.
Note that as $\vecphi(\Phi^n(\xx,t),t)$ is independent of $t$,
the integrand in the last expression is a polynomial of degree $d-1$ in $t$ where $d$ is the spatial dimension.
To integrate it exactly, when $d=2$, we can use the mid-point rule and when $d=3$ we can use Gauss quadrature.
Take $d=2$ as an example: The right hand side of the above equation equals
\begin{align}
  \Dt \left.\int_{\Omg^f_{h,t^n}} \!\!\vecphi(\Phi^n(\xx,t),t)
  \frac{\pa}{\pa t} \left|\frac{\pa \Phi^n(\xx,t)}{\pa \xx}\right| d\xx \right|_{t=t^{\np}}.  \label{gcl.6}
\end{align}
Now, let $\GG(\xx,t)=\frac{\pa \Phi^n(\xx,t)}{\pa\xx}$ and recall $\ww^{n}(\xx)=\pa_t\Phi^n(\xx,t)$.
\begin{align}
   \frac{\pa}{\pa t}\det \GG(\xx,t) &= \det\GG(\xx,t) \mbox{tr}\(\pa_t\GG(\xx,t) \GG(\xx,t)^{-1}\) \nonumber \\
 &= \det\GG(\xx,t) \mbox{tr}\( \frac{\pa \ww^{n}(\xx)}{\pa\xx} \GG(\xx,t)^{-1} \). \label{gcl.7}
\end{align}
Then, because of \qref{wn-st}, $\vecw^{n}(\Phi^n(\xx,t),t)=\ww^{n}(\xx)$.
Hence
\[
  \frac{\pa \ww^{n}(\xx)}{\pa\xx} = \frac{\pa \vecw^{n}(\Phi^n(\xx,t),t)}{\pa\xx}
  =\left.\frac{\pa \vecw^{n}(\yy,t)}{\pa\yy}\right|_{\yy=\Phi^n(\xx,t)} \frac{\pa \Phi^n(\xx,t)}{\pa \xx}.
\]
So we can continue \qref{gcl.7} and obtain
\begin{align*}
   \det\GG(\xx,t) \mbox{tr}\( \frac{\pa \ww^{n}(\xx)}{\pa\xx} \GG(\xx,t)^{-1} \)
 = \det\GG(\xx,t) \tr \(\frac{\pa \vecw^{n}(\yy,t)}{\pa\yy}\Big|_{\yy=\Phi^n(\xx,t)}\).
\end{align*}
Putting all together, we have
\begin{align*}
 & \int_{\Omg^f_{h,t^n}}\!\!\vecphi(\yy,t^n)d\yy-
   \int_{\Omg^f_{h,t^{n\!-\!1}}}\!\!\vecphi(\yy,t^{n-1})d\yy \\
 = & \Dt \left.\int_{\Omg^f_{h,t^n}} \!\!\vecphi(\Phi^n(\xx,t),t)
   \det\(\frac{\pa \Phi^n(\xx,t)}{\pa\xx}\) \; \tr \(\frac{\pa \vecw^{n}(\yy,t)}{\pa\yy}\Big|_{\yy=\Phi^n(\xx,t)}\) d\xx \right|_{t=t^{\np}}.
\end{align*}
After a change of variable $\yy=\Phi^n(\xx,t^\np)$ and using \qref{Omgt}, we obtain \qref{gcl.4}.
\end{proof}


Because in the above proof we only use $\frac{d}{dt}\vecphi_i^{f,n}(\Phi^n(\xx,t),t)=0$,
the $\vecphi_i^{f,n}$ in \qref{gcl.4} can be changed to $\vecphi_i^{f,n} \vecphi_j^{f,n}$. Therefore, we have the following Corollary:
\begin{corollary} \label{L.int.unun}
When $\Omg^f_{h,t^n}\subset\RR^2$, we have
\begin{equation} \label{gcl.5}
  \int_{\Omg^f_{h,t^n}}\!\!|\uu_h^n(\xx)|^2d\xx-
 \int_{\Omg^f_{h,t^{\npp}}}\!\!\! |\vecu_h^n(\yy,t^{\npp})|^2d\yy = \Dt
 \int_{\Omg^f_{h,t^{\np}}} \!\!\! |\vecu_h^n(\yy,t^\np)|^2 \div \vecw^{n}(\yy,t^\np) d\yy
\end{equation}
where $\uu_h^n$ is any function defined on $\Omg^f_{h,t^n}$ and $\vecu_h^n$ is
its backward in time extension $(\ref{ff-st}a)$. Using Gauss quadrature in time, we have similar
formula when $\Omg^f_{h,t^n}\subset\RR^3$.
\end{corollary}


Now we study the relation between the mesh velocity $\ww^{n}$ and fluid velocity $\uu_h^{n-1}$.
Note that $\ww^{n}$ is defined on $\Omg^f_{h,t^{n}}$ while
$\uu_h^{n-1}=\sum_{i=1}^G \uu_{h,i}^{n-1}\phi^{f,n-1}_i$ is defined
on $\Omg^f_{h,t^{n-1}}$.
But by Proposition~\ref{P.basis.extend}, $\uu_h^{n-1}=\sum_{i=1}^G \uu_{h,i}^{n-1}\vecphi^{f,n}_i(\cdot,t^{n-1})$.
So, we can introduce the forward in time extension defined in the space-time domain $Q_{[t^{n-1},t^n]}$:
\begin{equation} \label{vecu}
  \lvecu_h^{n-1}(\yy,t)=\sum_{i=1}^G \uu_{h,i}^{n-1}\vecphi^{f,n}_i(\yy,t).
\end{equation}

\begin{lemma} \label{L.wu}
Note that $\pa\Omg^f_{h,t}\backslash(\Sigma_1\cup\Sigma_2)$ is the FS interface.
For any $t\in[t^{n-1},t^n]$,
\begin{equation} \label{w=u}
  \vecw^{n}(\cdot,t)=\lvecu_h^{n-1}(\cdot,t)  \qquad \mbox{ on } \quad \pa\Omg^f_{h,t}\backslash(\Sigma_1\cup\Sigma_2).
\end{equation}
\end{lemma}
\begin{proof}
From \qref{Phit.2} and \qref{wwnt}, we know $\ww^n(\xx)=\sum_{i=1}^G \phi_i^{f,n}(\xx)\frac{\aa_i^n-\aa_i^{n-1}}{\Dt}$.
Hence $\vecw^{n}(\yy,t) = \sum_{i=1}^G \vecphi_i^{f,n}(\yy,t)\frac{\aa_i^n-\aa_i^{n-1}}{\Dt}$.
Comparing it with \qref{vecu}, we are left to show that $\uu_{h,i}^{n-1}=\frac{\aa_i^n-\aa_i^{n-1}}{\Dt}$
if $i$ is the index of a grid point on the FS interface.

In the later discussion (see the very last condition in the definition of $V^{\vphi_h^n}_{h,t^n}$ in \qref{Vh}),
we will see that fluid velocity $\uu_h^{n-1}$ and solid velocity $\vv_h^{n-1}$ agree at the grid points on the FS interface.

Let us use $\aa^{n-k}_{i_b}$ to denote grid points on
$\pa\Omg^f_{h,t^{n-k}}\backslash(\Sigma_1\cup\Sigma_2)$ for $k=0,1$.
So, by the way we construct $\Omg^f_{h,t^n}$ (recall $\vphi_h^n=\vphi_h^{n-1}+\Dt\vv_h^{n-1}$),
we know $\aa^{n-1}_{i_b}$ moves to $\aa^{n-1}_{i_b} + \Dt \uu^{n-1}_{h,i_b}$, i.e., $\aa^n_{i_b} = \aa^{n-1}_{i_b} + \Dt \uu^{n-1}_{h,i_b}$.
Therefore  $\uu^{n-1}_{h,i_b}= \frac{\aa_{i_b}^n-\aa_{i_b}^{n-1}}{\Dt}$.
\end{proof}
%
%
%
%
%

\subsection{The complete scheme}
From now on, we use the notation
\begin{gather}
 \<\cdot,\cdot\>_{(t)}=\<\cdot,\cdot\>_{\Omg^f_{h,t}}, \qquad \|\cdot\|_{(t)}=\|\cdot\|_{L^2(\Omg^f_{h,t})}
 =\<\cdot,\cdot\>^{\frac12}_{(t)}, \\
 \int_{\Omg^f_{h,t}}g(x)dx = \int_{(t)} g(x)dx.
\end{gather}
Inspired by \qref{onefield.conserve}, we propose the following first order scheme:
Suppose we are given $\uu_h^{n-1}$, $\vv_h^{n-1}$, $\vphi_h^{n-1}$ and $\TT^f_{h,t^{n-1}}$.
For the next moment $t^n$, first define
$\vphi_h^n$ by \qref{vphihn}:
\[
 \vphi_h^n = \vphi_h^{n-1} + \Dt \vv_h^{n-1}.
\]
Next, we construct $\TT^f_{h,t^n}$ and then the intermediate mesh $\TT^f_{h,t^{n-\frac12}}$ following
the discussion in the beginning of Section~\ref{S.fluid} and then Section~\ref{S.interMesh}.
Then we explicitly construct mesh velocity $\ww^n$ using \qref{wwnt} which is a finite element function defined on $\TT^f_{h,t^n}$.
Define the $\PP_m/\PP_{m-1}/\PP_m$ Lagrange finite element space $V^{\vphi_h^{n}}_{h,t^n}$ as follows:
\begin{align}
 V^{\vphi_h^{n}}_{h,t^n} = \Big\{& (\uu_h;p_h;\vv_h): \quad \uu_h\in C^0(\Omg^f_{h,t^n}), \quad p_h\in C^0(\Omg^f_{h,t^n}), \quad
 \vv_h\in C^0(\Omg^s), \nonumber \\
 & \forall \; T^{f,n}_j\in\TT^f_{h,t^n}, \; \uu_h\circ \Psi^{f,n}_{j} \in \PP_m(\hat{T}), \;
                                                 p_h\circ\Psi^{f,n}_{j} \in \PP_{m\!-\!1}(\hat{T}), \nonumber \\
 & \forall \; T^s_{j}\in\TT^s_{h}, \; \vv_h\circ \Psi^s_{j} \in \PP_{m}(\hat{T}),   \nonumber \\
 & \forall \; \mbox{ grid point }\zz_i \text{ of }\TT^\Gam_h, \; \uu_h(\vphi_h^{n}(\zz_i))=\vv_h(\zz_i) \Big\}. \label{Vh}
\end{align}
Here $\TT^\Gam_h$ is the mesh of $\Gam=\pa\Omg^f_{(0)}\cap \pa\Omg^s$. $\TT^s_{h}$ and $\TT^f_{h,t^n}$ are the meshes
of $\Omg^s$ and $\Omg^f_{h,t^n}$ respectively.
$\Psi^{f,n}_{j}$ is the mapping from the reference triangle $\hat{T}$ to the
$j$th physical triangle of the fluid domain at time $t^n$ which
is denoted by $T^{f,n}_{j}\in \TT^f_{h,t^n}$.
$\Psi^s_{j}$ is defined similarly, but for triangle $T^s_j\in \TT^s_h$.

Now, find $(\uu_h^n;p_h^n;\vv_h^n)\in V^{\vphi_h^{n}}_{h,t^n}$ with $\uu_h^n|_{\Sigma_1}=\uu_b$ and
$\vv_h^n|_{\Sigma_3}=\pa_t\vphi_b$ so that
for any finite element triple $(\tphi^{f,n};q^f;\tphi^s)\in V^{\vphi_h^{n}}_{h,t^n}$ with $\tphi^{f,n}|_{\Sigma_1}=0$ and
$\tphi^s|_{\Sigma_3}=0$,
\begin{align}
 & \frac{\rho^f}{\Dt}\(\<\uu_h^n,\tphi^{f,n}\>_{(t^n)}\!-\<\uu_h^{n\!-\!1},\vectphi^{f,n}(\cdot,t^{n-1})\>_{(t^\npp)}\)
 \nonumber \\
 & + \rho^f\<\div\(\vecu_h^n(\cdot,t^{\np})\otimes\(\lvecu_h^{n-1}(\cdot,t^\np)\!-\! \vecw^{n}(\cdot,t^\np)\)\),
                                                                        \; \vectphi^{f,n}(\cdot,t^{\np}) \>_{(t^\np)} \nonumber \\
 & - \rho^f\<\frac12 \(\div \lvecu_h^{n-1}(\cdot,t^\np)\) \vecu_h^n(\cdot,t^{\np}), \; \vectphi^{f,n}(\cdot,t^{\np}) \>_{(t^\np)} \nonumber \\
 & + \<\sig^f(\uu_h^n,p_h^n),\grad \tphi^{f,n}\>_{(t^n)} - \<\div\uu_h^{n}, q^f \>_{(t^n)} \nonumber \\
 & + \<\rho^s \frac{\vv_h^{n}-\vv_h^{n-1\!}}{\Dt}\! ,\tphi^s\>_{\Omg^s} + \<\sig^s(\vphi_h^n+\Dt\vv_h^n),\grad\tphi^s\>_{\Omg^s} \nonumber \\
 = & \<\rho^f\gg^f,\tphi^{f,n}\>_{(t^n)} +
     \<\sig_b^f,\tphi^{f,n}\>_{\Sigma_2}
     + \<\rho^s\gg^s,\tphi^s\>_{\Omg^s}
     +\<\sig_b^s,\tphi^s\>_{\Sigma_4}. \label{fluid.1}
\end{align}
The $\vectphi^{f,n}(\cdot,t)$ in \qref{fluid.1} is the backward in time extension of
vector basis function $\tphi^{f,n}$ by \qref{phi-st} (with obvious extension to vectors).
$\vecu_h^n(\cdot,t)$ is the backward in time extension of $\uu^{n}_h$ by \qref{ff-st}.
$\lvecu_h^{n-1}(\cdot,t)$ is the forward in time extension of $\uu^{n-1}_h$ by \qref{vecu}.
The technique of adding the term containing $-\frac12 (\div \lvecu_h^{n-1}(\cdot,t^\np))$ is standard and is initiated by \cite{Te}.
The above scheme is for $d=2$. When $d=3$, the 2nd and the 3rd lines of \qref{fluid.1} should be changed to
\begin{align}
  & +\frac{\rho^f}{2} \sum_{\ell=1}^2 \<\div\(\vecu_h^n(\cdot,t^{n_\ell})
    \otimes\(\lvecu_h^{n-1}(\cdot,t^{n_\ell})\!-\! \vecw^{n}(\cdot,t^{n_\ell})\)\),
    \; \vectphi^{f,n}(\cdot,t^{n_\ell}) \>_{(t^{n_\ell})} \nonumber \\
 &  -\frac{\rho^f}{2} \sum_{\ell=1}^2  \<\frac12 \(\div \lvecu_h^{n-1}(\cdot,t^{n_\ell})\) \vecu_h^n(\cdot,t^{n_\ell}),
    \; \vectphi^{f,n}(\cdot,t^{n_\ell}) \>_{(t^{n_\ell})} \label{fluid.1.3d}
\end{align}
where $\{t^{n_1},t^{n_2};\frac12,\frac12\}$ forms the two-point Gauss quadrature on $[t^{n-1},t^n]$ (Lemma~\ref{L.gcl.1}).


Let $\uu_h^n=\sum_{i=1}^{G_1} \uu^n_{h,i} \phi^{f,n}_i$, $p_h^n=\sum_{j=1}^{G_2} p^n_{h,j} q^{f}_j$
and $\vv_h^n=\sum_{k=1}^{G_3} \vv^n_{h,k} \phi^{s}_k$. \qref{fluid.1} leads to a system of equations
for $\{\uu^n_{h,i},p^n_{h,j},\vv^n_{h,k}\}$ and is linear for the fluid variables.
Proposition~\ref{P.basis.extend} tells us that the assembling of the load vectors and various matrices for each term in \qref{fluid.1}
uses only standard finite element basis functions defined on the corresponding mesh
indicated by the subscripts $t^n$, $t^\np$ or $t^{\npp}$ respectively.

Lastly, we would like to stress that the last equality condition in the definition of
$V^{\vphi^n_h}_{h,t^n}$ (\qref{Vh}) is trivial to enforce and will not complicate the programming:
When assembling the matrices and vectors, we simply need to equate the global index
of the fluid basis function $\tphi^{f}$ that is associated with $\vphi_h^n(\zz_i)$
with the global index of the solid basis function $\tphi^s$ that is associated with $\zz_i$,
for all grid point $\zz_i\in \TT^\Gam_h$.

\section{Stability} \label{S.conservative.stab}
Now, we are ready to prove the stability of scheme \qref{fluid.1}.
\begin{theorem} \label{T.stab.fluid}
Assume $W(\FF)$ is convex in \qref{strain} and assume $\uu_b=0$, $\Sigma_2=\emptyset$ and $\pa_t \vphi_b=0$.
Then for any $\Dt$ that satisfies \qref{stab.assum.1},
\begin{align}
 &\frac{\rho^f}{2\Dt}\(\|\uu^n_h\|^2_{(t^n)} -\|\uu^{n-1}_h\|^2_{(t^{n-1})}\)
  + \frac{\rho^f\nu^f}{2}\|\grad \uu^n_h+\grad \uu^{n,\top}_{h}\|^2_{(t^n)} \nonumber \\
  & + \frac{1}{\Dt}\(\IIs(\vv_h^n;\vphi_h^{n}+\Dt\vv_h^n)-\IIs(\vv_h^{n-1};\vphi_h^{n-1}+\Dt\vv_h^{n-1})\) \nonumber \\
 \le & \<\rho^f\gg^f,\uu_h^n\>_{(t^n)} + \<\rho^s\gg^s,\vv_h^n\>_{\Omg^s}
  + \<\sig_b^s,\vv_h^n\>_{\Sigma_4} \label{stab.conservative}
\end{align}
where energy function $\IIs(\vv;\vphi)=\frac{\rho^s}{2}\|\vv\|^2_{\Omg^s} + \int_{\Omg^s}W(\grad\vphi)$.
In particular, the following quantity is uniformly
bounded in $n$ and the bound is independent of the fluid mesh velocity:
\[
  2\;\IIs(\vv_h^n;\vphi_h^{n}+\Dt\vv_h^n)+ \rho^f\|\uu^n_h\|^2_{(t^n)} +
  \sum_{\ell=1}^n \rho^f\nu^f\|\grad \uu^\ell_h+\grad \uu^{\ell,\top}_{h}\|^2_{(t^\ell)}\Dt.
\]
\end{theorem}
\begin{proof}
We have discussed the solid part (5th line in \qref{fluid.1}) in Section~\ref{S.solid.stab} which leads to \qref{solid.5}.
In \qref{fluid.1}, let $\tphi^{f,n}=\uu_h^n$, $q^f=-p_h^n$ and $\tphi^s=\vv_h^n$. Note that
$\vectphi^{f,n}(\cdot,t^{\np})$ becomes $\vecu^n_h(\cdot,t^\np)$. We obtain
\begin{align}
 & \frac{\rho^f}{\Dt}\|\uu^n_h\|^2_{(t^n)} -\frac{\rho^f}{2\Dt}\(\|\vecu^n_h(\cdot,t^{n-1})\|^2_{(t^\npp)}
     + \|\uu^{n-1}_h\|^2_{(t^\npp)}\) +\<\sig^f(\uu_h^n,0),\grad \uu^n_h\>_{(t^n)} \nonumber \\
   & + \rho^f\<\div\(\vecu_h^n(\cdot,t^{\np})\otimes\(\lvecu_h^{n-1}(\cdot,t^\np)\!-\!\vecw^{n}(\cdot,t^\np)\)\),\;
                                                                               \vecu^n_h(\cdot,t^\np) \>_{(t^\np)} \nonumber \\
   & - \rho^f\<\frac12 \(\div\lvecu_h^{n-1}(\cdot,t^\np)\) \vecu_h^n(\cdot,t^\np), \; \vecu^n_h(\cdot,t^\np) \>_{(t^\np)} \nonumber \\
   & + \frac{1}{\Dt}\IIs(\vv_h^n;\vphi_h^{n+1})-\frac{1}{\Dt}\IIs(\vv_h^{n-1};\vphi_h^{n}) \nonumber \\
  \le & \<\rho^f\gg^f,\uu_h^n\>_{(t^n)}
     + \<\rho^s\gg^s,\vv_h^n\>_{\Omg^s} + \<\sig_b^s,\vv_h^n\>_{\Sigma_4}.  \label{fluid.3}
\end{align}
Here, without loss of generality we consider $d=2$.
Now, look at the 2nd line in the above inequality
and recall $\div A = \pa_j A_{ij}$. After integration by part, we find it equals
\begin{align*}
  & - \rho^f\frac12\int_{(t^\np)} \(\lvecu_h^{n-1}(\cdot,t^\np)\!-\!\vecw^{n}(\cdot,t^\np)\)\cdot
                                                                               \grad |\vecu^n_h(\cdot,t^\np)|^2 \\
  & + \rho^f\int_{\pa \Omg^f\!\!\!{}_{h,t^{\np}}} \(\lvecu_h^{n-1}(\cdot,t^\np)\!-\!\vecw^{n}(\cdot,t^\np)\)\cdot\nn |\vecu^n_h(\cdot,t^\np)|^2.
\end{align*}
Integration by part once again, the 2nd line in \qref{fluid.3} becomes
\begin{align}
 & \rho^f \frac12\int_{(t^\np)} \div\(\lvecu_h^{n-1}(\cdot,t^\np)\!-\!\vecw^{n}(\cdot,t^\np)\)
                                                                               |\vecu^n_h(\cdot,t^\np)|^2 \nonumber \\
 & + \rho^f \frac12 \int_{\pa \Omg^f\!\!\!{}_{h,t^{\np}}}
    \(\lvecu_h^{n-1}(\cdot,t^\np)\!-\!\vecw^{n}(\cdot,t^\np)\)\cdot\nn |\vecu^n_h(\cdot,t^\np)|^2.
 \label{fluid.5}
\end{align}
Because of the assumption $\uu_b=0$ and $\Sigma_2=\emptyset$ and
Lemma~\ref{L.wu} (indeed, we only need their normal components equal), the boundary integral in \qref{fluid.5} vanishes.
The first half of the volume integral in \qref{fluid.5} cancels with the 3rd line of \qref{fluid.3}.
So, \qref{fluid.3} now becomes
\begin{align}
 & \frac{\rho^f}{\Dt}\|\uu^n_h\|^2_{(t^n)} -\frac{\rho^f}{2\Dt}\(\|\vecu^n_h(\cdot,t^{n-1})\|^2_{(t^{n-1})}
   + \|\uu^{n-1}_h\|^2_{(t^{n-1})}\) + \<\sig^f(\uu_h^n,0),\grad \uu^n_h\>_{(t^n)} \nonumber \\
 & - \frac{\rho^f}{2}\!\!\int_{(t^\np)} \div\(\vecw^{n}(\cdot,t^\np)\) |\vecu^n_h(\cdot,t^\np)|^2
   + \frac{1}{\Dt}\IIs(\vv_h^n;\vphi_h^{n+1})-\frac{1}{\Dt}\IIs(\vv_h^{n-1};\vphi_h^{n}) \nonumber \\
\le & \<\rho^f\gg^f,\uu_h^n\>_{(t^n)}
     + \<\rho^s\gg^s,\vv_h^n\>_{\Omg^s} + \<\sig_b^s,\vv_h^n\>_{\Sigma_4}. \label{fluid.4}
\end{align}
Then we make use of \qref{gcl.5} to handle the 2nd and the 5th terms.
Finally, we use $\<\sig^f(\uu_h^n,0),\grad \uu^n_h\>_{(t^n)}
=\frac{\rho^f\nu^f}{2}\|\grad\uu_h^n+\grad\uu_h^{n,\top}\|^2_{(t^n)}$ to conclude.
\end{proof}


\begin{rmk} \label{R.Remark1}
Our FSI solver works as follows: at $t^{n-1}$, we have $\uu_h^{n-1}$, $\vv_h^{n-1}$, $\vphi_h^{n-1}$.
First, we define $\vphi_h^{n}=\vphi_h^{n-1}+\Dt \vv_h^{n-1}$.
Then we assume \qref{stab.assum.1} is satisfied so that we can construct the fluid domain $\Omg^f_{h,t^{n}}$ with mesh $\TT^f_{h,t^n}$.
Then, on $\TT^f_{h,t^n}$ and $\TT^s_h$, we solve for $(\uu_h^{n};p_h^n;\vv_h^{n})$.
Immediately there comes a good news for the construction of $\Omg^f_{h,t^{n+1}}$ which is for the next step:
From \qref{stab.conservative}, even at $t^n$, we already know
\begin{align}
  & \frac{\mu^s}{4}\|\grad \eeta_h^{n+1}+\grad \eeta_h^{n+1,\top}\|_{\Omg^s}^2
  + \frac{\lam^s}{2}\|\div \eeta_h^{n+1}\|_{\Omg^s}^2
 + \frac{\rho^s}{2}\|\vv_h^{n}\|_{\Omg^s}^2 + \frac{\rho^f}{2} \|\uu_h^{n}\|_{(t^{n})}^2 \nonumber \\
\le & \frac{\mu^s}{4}\|\grad \eeta_h^{n}+\grad \eeta_h^{n,\top}\|_{\Omg^s}^2
  + \frac{\lam^s}{2}\|\div \eeta_h^{n}\|_{\Omg^s}^2
 + \frac{\rho^s}{2}\|\vv_h^{n-1}\|_{\Omg^s}^2 + \frac{\rho^f}{2} \|\uu_h^{n-1}\|_{(t^{n-1})}^2.
\end{align}
For simplicity, we have assumed linear elasticity
with $\eeta^{n+1}(\zz)=\vphi^{n+1}(\zz)-\zz$ and ignored the body forces and $\sig_b^s$.
So $\vphi_h^{n+1}$ is rather regular which makes the assumption \qref{stab.assum.1} less stringent
because the construction of $\Omg^f_{h,t^{n+1}}$ uses $\vphi_h^{n+1}(\Gam)$. Obviously, larger $\mu^s$
and $\lam^s$ would provide larger support for validating assumption \qref{stab.assum.1}. When $\mu^s=+\infty=\lam^s$,
our method solves fluid and rigid body interaction problem.
\end{rmk}

\subsection{Existence and uniqueness} \label{S.existence}
Certainly, before we ever discuss the stability, we need show the system \qref{fluid.1}
does have a solution.
As the mesh is determined explicitly and the convection term in the fluid is handled semi-implicitly,
the only nonlinear term is $\<\sig^s(\vphi_h^n+\Dt\vv_h^n),\grad\tphi^s\>_{\Omg^s}$.

If the solid is linear elastic (see \qref{lin.material}),
\qref{fluid.1} is indeed a linear system for $(\uu^n_h;p^n_h;\vv^n_h)$ and existence follows.
The uniqueness follows from the stability results:
When all the forcing terms vanish, by choosing $(\tphi^{f,n},q^f,\tphi^s)=(\uu^n_h;p^n_h;\vv^n_h)$, we know
$\uu^n_h=0$ and $\vv^n_h=0$. Then \qref{fluid.1} implies $\<p^n_h,\div \tphi^{f,n}\>_{(t^n)}=0$ for all $\tphi^{f,n}$
in the finite element space for fluid velocity and having zero boundary condition on $\Sigma_1 \subsetneq \pa\Omg^f_{h,t^n}$.
So, by the inf-sup condition, $p^n_h=0$.

For nonlinear solid with convex strain energy,
we need to assume $\uu_b=0$, $\Sigma_2=\emptyset$ and $\pa_t \vphi_b=0$.
Then the stability itself will imply existence by the following Lemma
\cite[Chap 2. Lemma 1.4]{Te}. The proof is a simple
application of Brouwer fixed point theorem.
\begin{lemma} \label{L.temam}
Let $X$ be a finite dimensional Hilbert space with inner product $(\cdot,\cdot)$ and norm $\|\cdot\|$.
Let T be a continuous mapping from $X$ into itself such that there is a constant $\beta$ so that
\[
  (T(\xi),\xi) >0 ,\qquad \forall \, \|\xi\|=\beta>0.
\]
Then there exists $\xi\in X$ with $\|\xi\|\le \beta$ such that $T(\xi)=0$.
\end{lemma}

\begin{theorem} \label{T.existence}
Assume $W(\FF)$ is convex in \qref{strain}.
If $\uu_b=0$, $\Sigma_2=\emptyset$ and $\pa_t \vphi_b=0$, then \qref{fluid.1} has a unique solution.
\end{theorem}
\begin{proof}
It is clear that we can define a continuous mapping $T_{n-1}(\uu_h^n;p^n_h;\vv^n_h)$
as well as the space $X$ and inner product $(\cdot,\cdot)$ so that
\qref{fluid.1} can be written as
\[
  \(T_{n-1}(\uu_h^n;p^n_h;\vv^n_h), (\tphi^{f,n};q^f;\tphi^s)\)=0.
\]
The subscript $n-1$ in $T_{n-1}$ means the mapping depends on
$(\uu_h^{n-1};\vv^{n-1}_h)$ and also boundary data and body forces.
From \qref{stab.conservative}, we know
$\(T_{n-1}(\uu^n_h;p^n_h;\vv^n_h),(\uu^n_h;p^n_h;\vv^n_h)\) \ge
\text{the left hand side of \qref{stab.conservative}}$
which is positive when $\|(\uu^n_h;p^n_h;\vv^n_h)\|$ is large enough.
Hence the existence follows from Lemma~\ref{L.temam}.

\medskip

Now consider the uniqueness.
Because of the convexity of $W(\FF)$ and $\vphi_h^n+\Dt\vv_h^n-(\vphi_h^n+\Dt\tvv_h^n)=\Dt(\vv_h^n-\tvv_h^n)$,
\begin{equation} \label{convex.stab}
\<\sig^s(\vphi_h^n+\Dt\vv_h^n)-\sig^s(\vphi_h^n+\Dt\tvv_h^n),\; \grad(\vv_h^n-\tvv_h^n)\>_{\Omg^s} \ge 0
\end{equation}
for any $\vv_h^n$ and $\tvv_h^n$. Now suppose we have two solution of \qref{fluid.1}.
Let us call them $(\uu_h^n;p^n_h;\vv^n_h)$ and $(\tuu_h^n;\tilde{p}^n_h;\tvv^n_h)$.
We take difference of the \qref{fluid.1}'s satisfied by $(\uu_h^n;p^n_h;\vv^n_h)$
and $(\tuu_h^n;\tilde{p}^n_h;\tvv^n_h)$ respectively and let the test function be
$(\uu_h^n-\tuu_h^n;p^n_h-\tilde{p}^n_h;\vv_h^n-\tvv^n_h)$.
From the stability results as well as \qref{convex.stab},
we immediately obtain $\uu_h^n=\tuu_h^n$ and $\vv_h^n=\tvv_h^n$.
After that, from the difference of the \qref{fluid.1}'s, we have
\[
  \<p_h^n-\tilde{p}^n_h,\div\tphi^{f,n}\>_{(t^n)}=0
\]
for all $\tphi^{f,n}$
in the finite element space for fluid velocity and having zero boundary condition on $\Sigma_1 \subsetneq \pa\Omg^f_{h,t^n}$.
So, by the inf-sup condition, $p^n_h-\tilde{p}^n_h=0$.
\end{proof}

\section{Numerical test} \label{S.num}

The finite element package we have implemented is in some sense an
upgraded version of iFEM due to Long Chen \cite{Ch,CZ}. iFEM is an adaptive
piecewise linear finite element package based on MATLAB. It uses a
beautiful data structure to represent the mesh and also provides
efficient MATLAB subroutines to manipulate the mesh.
In particular, local refinement and coarsening can
be done fairly easily. For our purposes, we have extended it to Taylor-Hood
isoparametric Lagrange elements $\PP_m/\PP_{m-1}$ with $m=2,...,5$. The finite element mesh
is generated by the
DistMesh of Persson and Strang \cite{PS}.


We present two numerical tests: (I) The first case is Navier-Stokes flow past
a linear elastic semi-cylinder. See Fig.~\ref{F.1} for an illustration.
The semi-cylinder is placed inside a channel and is attached to the floor.
The size of the channel is $[0,6.5]\times[-0.5,1]$. The cylinder is centered at $[1.5,-0.5]$ and has radius 0.5.
The inflow from the left is prescribed by $(u,v)=\(g(t)(1+2y)(1-y),0\)$ where $g(t)=\frac{1-\cos(\frac{\pi}{2}t)}{2}$ when $t\le 2$
and $g(t)=1$ when $t\ge 2$. At the outflow boundary, we use $\sig^f\nn=0$ as the boundary condition \cite{Liu4}.
(II) The second case is a Navier-Stokes flow enforced vibrating bar. This problem is proposed by \cite{TH2}.
See Fig.~\ref{F.4} for an illustration.
A rigid cylinder centered at $(0.2,0.2)$ with radius $0.05$ is fixed inside a channel of size $[0,2.5]\times[0,0.41]$.
A horizontal St. Venant-Kirchhoff bar with length 0.35 and width 0.02 is attached to the rigid cylinder.
The surface where they touch is curved.
The center of the cylinder and the center of the bar
have the same height initially. The material point on the tail of the bar which is initially at $(0.6,0.2)$ is called $P_{\text{tail}}$.
The inflow velocity of the channel is $(u,v)=g(t)(\frac{12}{0.1681}y(0.41-y), 0)$ where the same $g(t)$ as in case (I) is used.
We also use the same outflow boundary condition as in case (I).
The physical parameters for these two test problems are listed in Table~\ref{T.1}.
\begin{table}[h]
\begin{center}
\begin{tabular}{|r|r|r|r|r|r|r|r|r|}
\hline
   & solid type & $\rho^f$ & $\nu^f$ & $\rho^s$ & $\mu^s$ & $\lam^s$ & $\gg^f$ & $\gg^s$ \\
\hline
case (I) & linear & 1 & 1 & 1 & 50 & 500 & (0,0) & (0,0) \\
\hline
case (II) & St.Venant-Kirchhoff & 1 & 0.001 & 1 & 2000 & 8000 & (0,$-2$) & (0,$-2$) \\
\hline
\end{tabular}
\end{center}
\caption{Physical parameters. (Flow is incompressible Navier-Stokes.)}\label{T.1}
\end{table}
The scheme we tested is \qref{fluid.1} except that the
St. Venant-Kirchhoff material in case (II) is treated semi-implicitly by \qref{solid.lin.3}
for efficiency.

We use case (I) for both stability and accuracy check.
To verify the stability, we use $\PP_2/\PP_1/\PP_2$ elements and take $\Dt=1$ to integrate to $t=10$.
The computational mesh is shown in Fig.~\ref{F.1}.
In the captions, we state parameters of the meshes
where $h_{\max}$ and $h_{\min}$ are the sizes of the largest and smallest edges.
The computational domain for the fluid will change but the computational domain for the solid will remain the same.
When $\Dt=1$, by $t=4$ (so after 4 iterations if doing time matching), the system has almost reached steady state.
The CPU time in that situation is about 12 seconds on an IBM Thinkpad laptop with 3G memory.
Even though it has intel Core 2 Duo CPU @ 2.8 GHz, the Matlab is run on a single thread mode.
We compare our result with result from a domain decomposition approach
and find that they agree rather well (see Fig.~\ref{F.3}). 
Then we verify the first order temporal accuracy of \qref{fluid.1} also using case (I).
We do not have a closed form for the exact solution and so we compute with a very small $\Dt$
and use the result as the ``exact" solution to do the accuracy check.
The results are listed in Table~\ref{Table.accu} and from that we see clean first order accuracy in time
(see the numbers put insider the bracket in Table~\ref{Table.accu}).
\begin{table}[htpb]
\begin{center}
\begin{tabular*}{1\textwidth}{@{\extracolsep{\fill}}|c||c|c|c|c|} \hline
$E$ \quad $\backslash$ \quad $\Dt$ & $0.1$ & $0.05$ & $0.025$ & $0.0125$ \\ \hline

$\|\varphi_1-\varphi_{1,\Dt}\|_{L^2(\Omg^s)}$ & $-2.48$ & $-2.78$ (0.988)  & $-3.08$ (0.998)  & $-3.38$ (1) \\
$\|\varphi_1-\varphi_{1,\Dt}\|_{L^\infty(\Omg^s)}$ & $-2.02$ & $-2.32$ (0.986)  & $-2.62$ (0.997)  & $-2.92$ (1) \\
$\|\grad(\varphi_1-\varphi_{1,\Dt})\|_{L^2(\Omg^s)}$ & $-1.89$ & $-2.19$ (0.99)  & $-2.49$ (0.999)  & $-2.79$ (1) \\
$\|\grad(\varphi_1-\varphi_{1,\Dt})\|_{L^\infty(\Omg^s)}$ & $-1.11$ & $-1.44$ (1.12)  & $-1.77$ (1.08)  & $-2.08$ (1.05) \\
\hline
$\|\varphi_2-\varphi_{2,\Dt}\|_{L^2(\Omg^s)}$ & $-3.44$ & $-3.75$ (1.05)  & $-4.05$ (1.01)  & $-4.35$ (0.98) \\
$\|\varphi_2-\varphi_{2,\Dt}\|_{L^\infty(\Omg^s)}$ & $-2.93$ & $-3.25$ (1.07)  & $-3.56$ (1.03)  & $-3.87$ (1.01) \\
$\|\grad(\varphi_2-\varphi_{2,\Dt})\|_{L^2(\Omg^s)}$ & $-2.78$ & $-3.09$ (1.05)  & $-3.4$ (1.01)  & $-3.7$ (0.993) \\
$\|\grad(\varphi_2-\varphi_{2,\Dt})\|_{L^\infty(\Omg^s)}$ & $-1.8$ & $-2.02$ (0.714)  & $-2.26$ (0.797)  & $-2.53$ (0.9) \\

\hline
\end{tabular*}
\end{center}
\caption{First order accuracy in time:
$\log_{10} E$ (and local order $\alpha$) vs $\Dt$.
$\alpha = \frac{\log_{10} (E_{k-1}/E_k)}{\log_{10} (\Dt_{k-1}/\Dt_k)}$.
We use $\PP_5/\PP_4/\PP_5$ isoparametric finite element mesh which is a global refinement (means $h\to h/2$) of the mesh shown in Fig.~\ref{F.1}.
We integrate to $t=1$ using $\Dt=[0.1,0.05,0.025,0.0125]$ and then
compare the resulting structure position (called $(\varphi_{1,\Dt},\varphi_{2,\Dt})$) with the
result using $\Dt=5\times 10^{-5}$ (called $(\varphi_{1},\varphi_{2})$).}
\label{Table.accu}
\end{table}
Table~\ref{Table.accu} does not show the error right at the interface. So, in Fig.~\ref{F.interface},
we plot the interface positions obtained with different $\Dt$ and plot the error of the
position vector on the FS interface $\Gam$ (which is a half circle and hence
are labeled by the angle $\theta\in[0,\pi]$).
From the right plot of Fig.~\ref{F.interface}, we can also see clearly that when $\Dt$
is reduced by half, the error decreases by half.

For case (II), we use $\PP_3/\PP_2/\PP_3$ elements and take $\Dt=0.0005$ to integrate to $T=8$.
The computational mesh is shown in Fig.~\ref{F.4}.
Some snap-shots of the results are shown in Fig.~\ref{F.5}.
The time step for case (II) is taken to be small merely for accuracy purpose as our scheme is first order accurate in time.
We can take much larger time step. Indeed, we have taken $\Dt=0.1$ and integrate to $t=20$ for case (II) using
the same mesh. The only problem with $\Dt=0.1$ is accuracy: For example, by $t=8$ the bar with
$\Dt=0.1$ just starts to vibrate while the bar with $\Dt=0.0005$ has already reached its periodic vibrating stage.
For case (II), we record the lift and drag forces as well as
the position of $P_{\text{tail}}$. In our non-dimensionalized equations, the lift and drag forces are
the $x$ and $y$ components of $1000\int_{S} \sig^f\nn ds$ where $S$ is the surface of the cylinder+bar and the $\nn$
is the outward normal with respect to the cylinder+bar.
According to \cite{TH2}, the lift and drag forces are $[-147.56,152.00]$ and $[434.64,479.96]$ respectively,
and the displacements in the $x$ and $y$ directions of $P_{\text{tail}}$ are
$[-5.22,-0.16]\times 10^{-3}$ and $[-32.90,35.86]\times 10^{-3}$ respectively.
Here we use interval $[a,b]$ to indicate the range of a periodic oscillating quantity.
We stress that we have used our favorite traction type open boundary condition
($\sig^f\nn=-p_g\nn$ where $p_g$ satisfies $\grad p_g = \rho^f\gg^f$)
which is physical and allows the bar to contract or expand freely. Moveover, it
nails down the arbitrary constant in the fluid pressure \cite{Liu4}.
If one prescribes the outflow profile which is the same as the inflow profile,
the volume of the bar will keep the same and
the arbitrary constant in the pressure will be determined by this constraint.
\cite{TH2} does not state the open boundary condition it uses.
If \cite{TH2} uses Dirichlet type open boundary condition, we expect our results be slightly different from those of \cite{TH2}.

\begin{figure}[hbtp]
\centerline{
\includegraphics[width=2.8in]{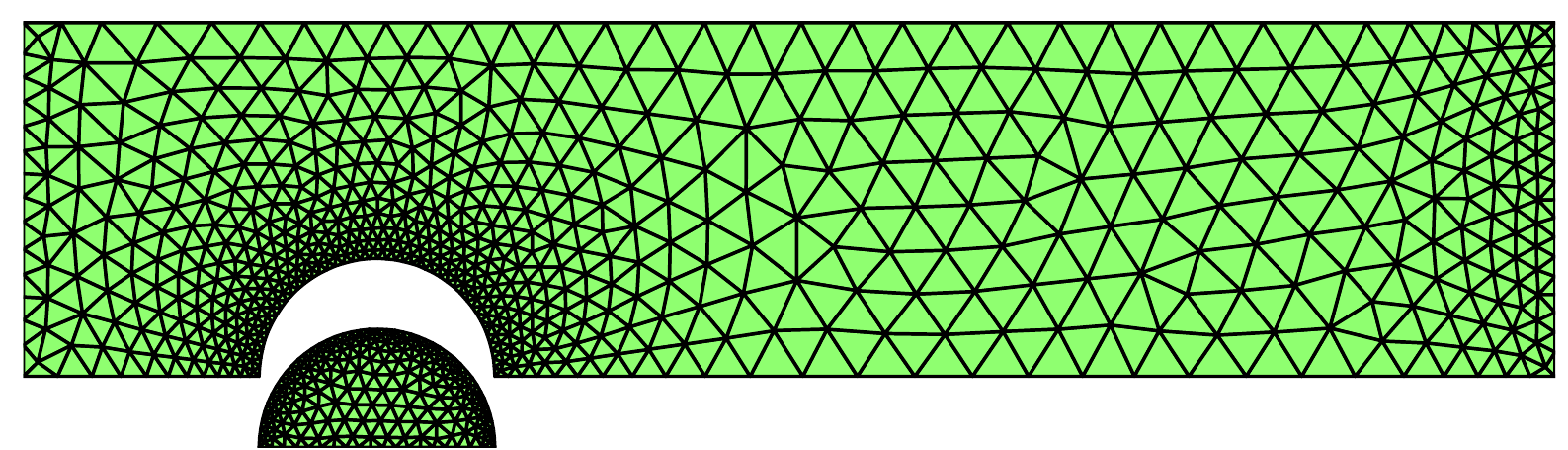}
\includegraphics[width=2.8in]{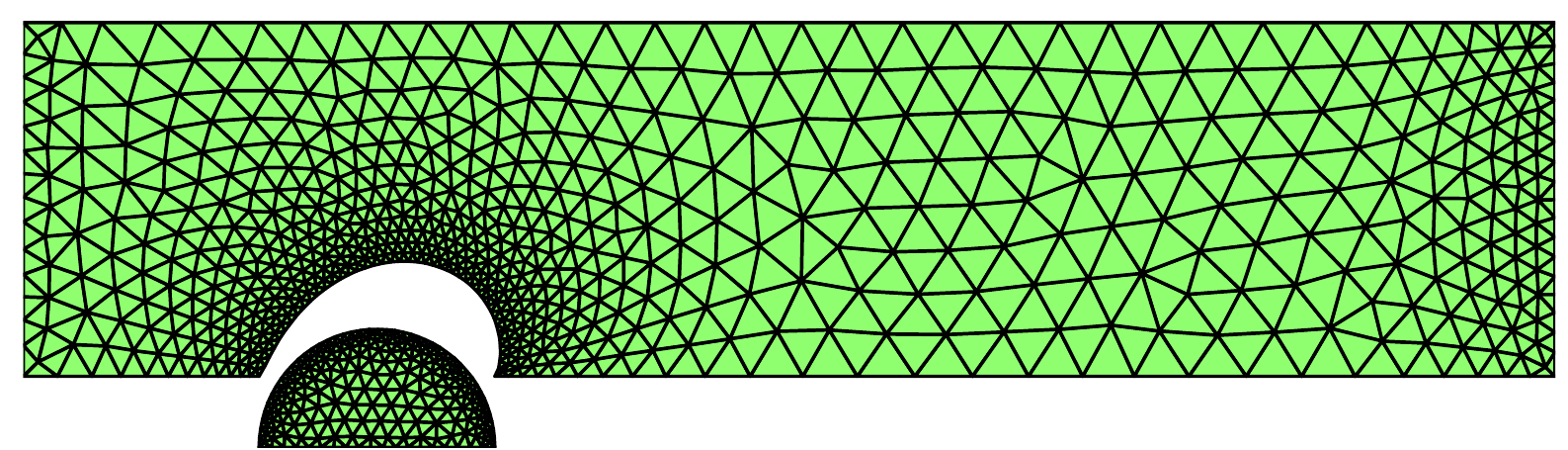}
}
\caption{Computational meshes at $t=0$ and then $t=10$ for case (I).
$\PP_2/\PP_1/\PP_2$. $h^f_{\min}=0.0290$. $h^f_{\max}=0.2743$.
$h^s_{\min}=0.0202$. $h^s_{\max}=0.0889$. ($h$ is the size of the edge. $f$ means fluid and $s$ means solid.)
1239 triangles for the fluid mesh and 370 triangles for the solid mesh.
$\Dt=1$.
}
\label{F.1}
\end{figure}

\begin{figure}[h]
\begin{minipage}{2.2in}
\centerline{
\includegraphics[width=2.2in]{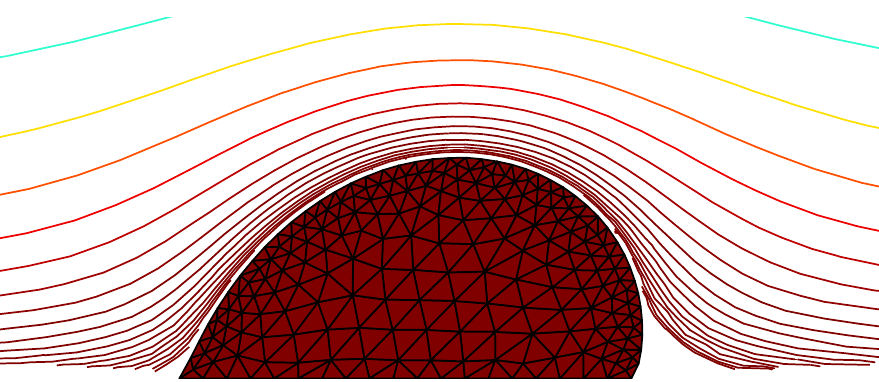}}
\bigskip
\end{minipage}
\begin{minipage}{3in}
\centerline{
\includegraphics[width=1.5in]{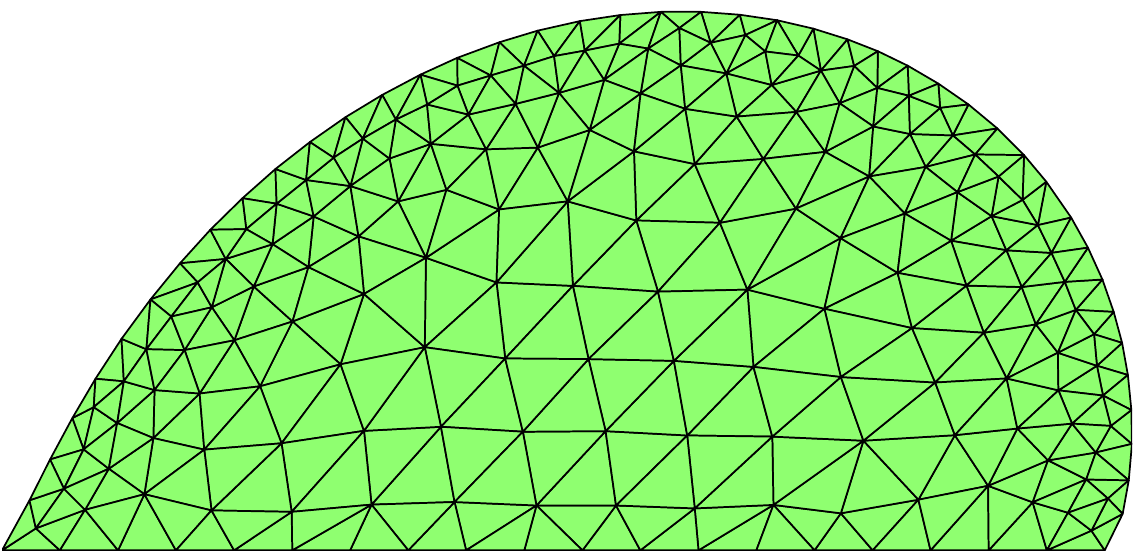}
\includegraphics[width=1.5in]{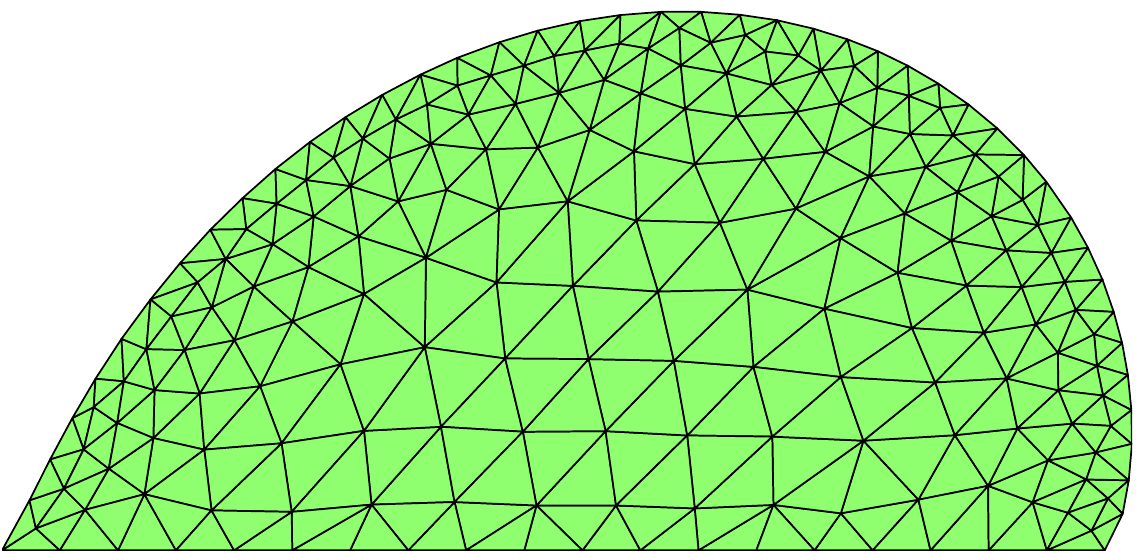}}
\end{minipage}
\caption{Left: Streamline around the deformed semi-cylinder at
$t=10$ for case (I). Middle: Structure position at $t=10$ using
domain decomposition approach with $\Dt=2\times 10^{-4}$. Right:
Structure position at $t=10$ using combined field formulation with
$\Dt=1$.} \label{F.3}
\end{figure}

\begin{figure}[hbtp]
\centerline{
\includegraphics[width=4.8in]{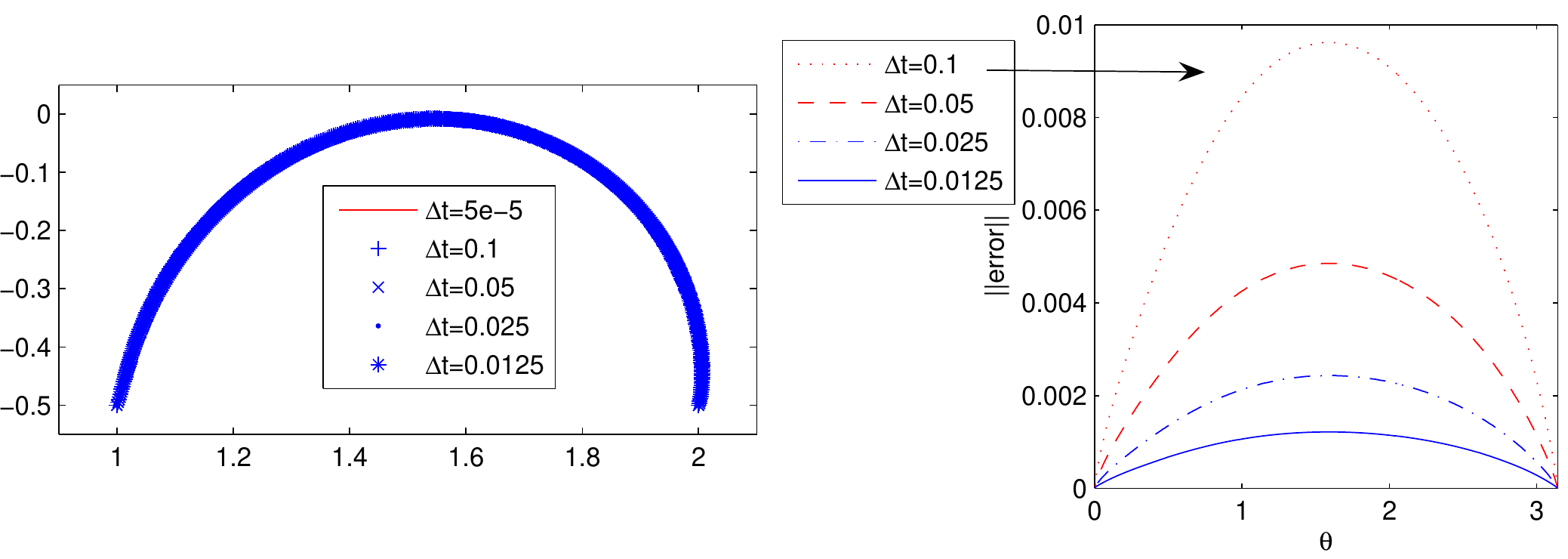}
}
\caption{Interface positions and interface errors at $t=1$ for case (I).
Left: The interface positions with different $\Dt$ (they coincide on the plot).
Right: The error on the interface, $\|\vphi(\zz,1)-\vphi_h(\zz,1)\|_{\ell^2}$ v.s. $\theta$.
Here $\theta=\theta(\zz)$ is angle between $\zz-(1.5,-0.5)$ and the positive $x$-axis.
Recall that $\zz$ is the material point on the interface $\Gam$ which is a semi-circle
for case (I) and $(1.5,-0.5)$ is the center of the semi-circle.
See Table~\ref{Table.accu} for the error on the whole $\Omg^s$.} \label{F.interface}
\end{figure}

\begin{figure}[hbtp]
\centerline{
\includegraphics[width=4.5in]{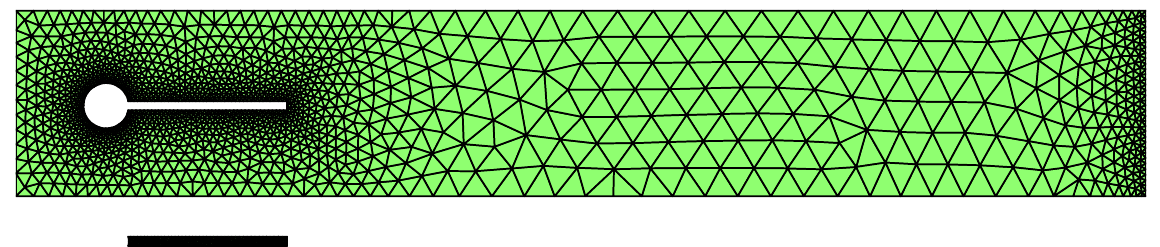}}
\centerline{
\includegraphics[width=4.5in]{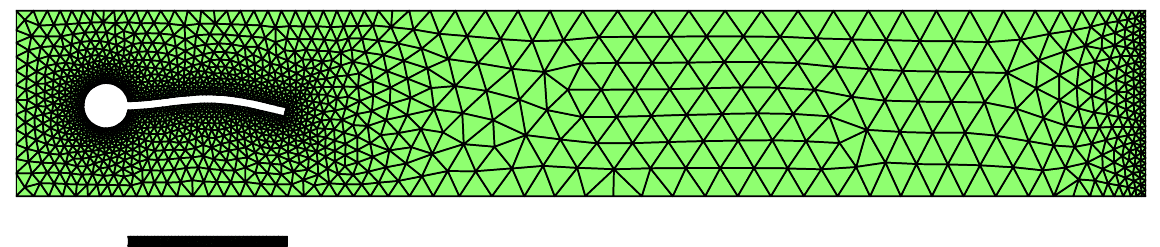}}
\caption{Computational meshes at $t=0$ and then $t=8$ for case (II). $\PP_3/\PP_2/\PP_3$. $h^f_{\min}=0.005$. $h^f_{\max}=0.088$.
$h^s_{\min}=0.0034$. $h^s_{\max}=0.01$. ($h$ is the size of the edge.)
2727 triangles for the fluid mesh and 274 triangles for the solid mesh.
$\Dt=0.0005$.}
\label{F.4}
\end{figure}

\begin{figure}[hbtp]
\centerline{
\includegraphics[width=4.8in]{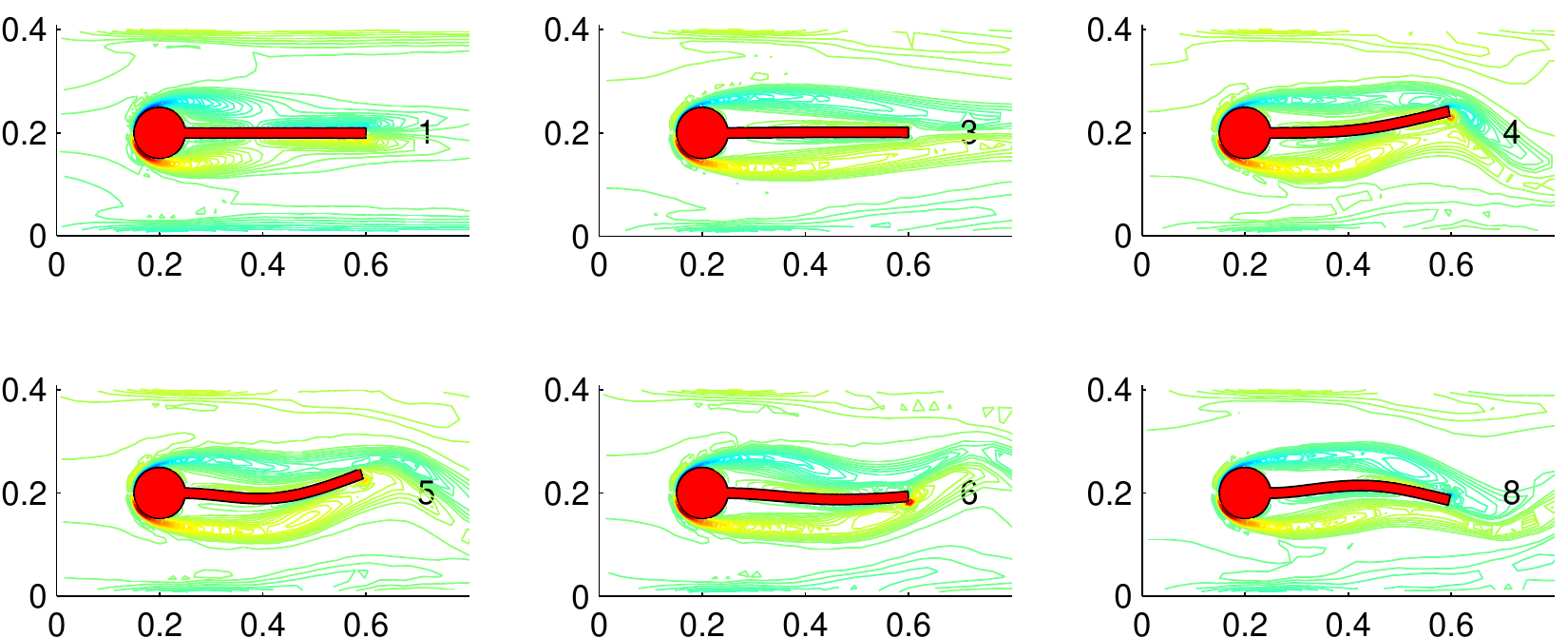}}
\caption{Case (II): Vorticity of the flow and position of the bar at different time. $t=[1,3,4,5,6,8]$. }
\label{F.5}
\end{figure}

\begin{figure}[hbtp]
\centerline{
\includegraphics[width=2.4in]{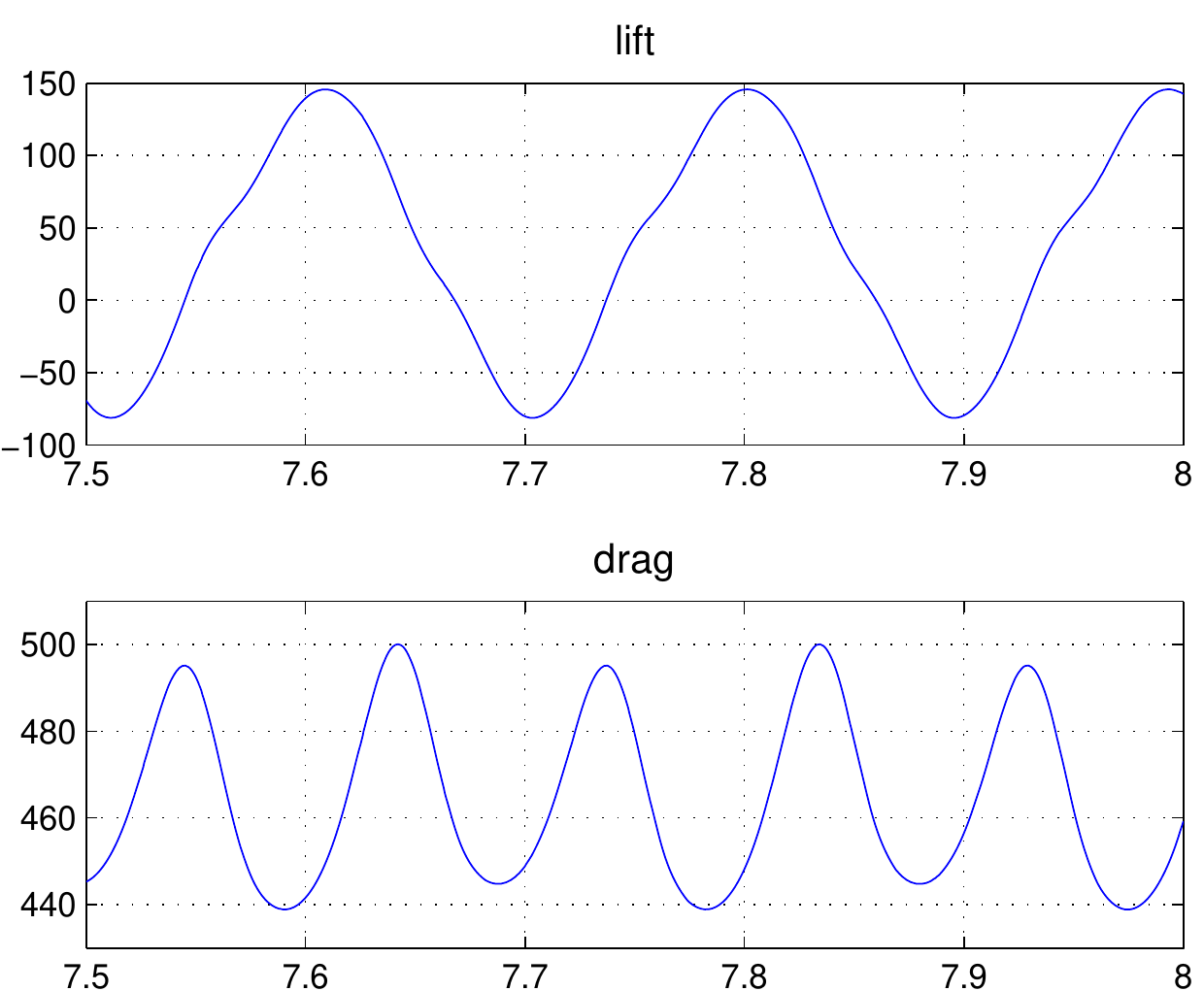}
\includegraphics[width=2.4in]{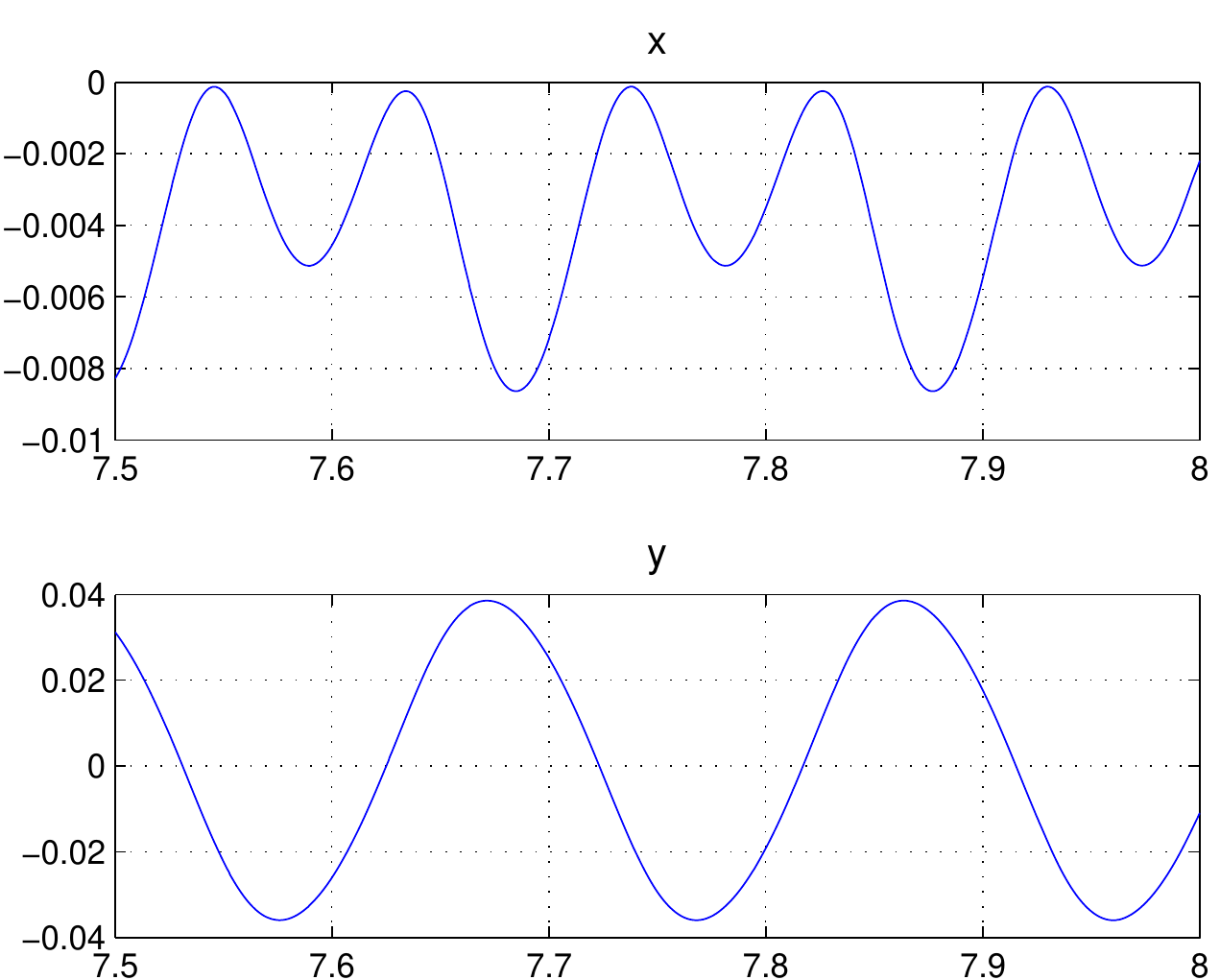}}
\caption{Case (II): Left plot: Lift and drag forces on the cylinder+bar.
Right plot: Displacement of the center point on the tail of the
bar.} \label{F.6}
\end{figure}

\section*{Acknowledgments}
The author's work is supported by the start-up
grant from the National University of Singapore with grant number
R-146-000-129-133. The author thanks Dr. Stuart Antman and Dr. Jian-Guo Liu
for the teaching of continuum mechanics.
The author thanks Dr. Long Chen for the finite element package iFEM
based on which the finite element calculations in this paper are done.
The author thanks Dr. Richard Kollar for helpful discussions.


\end{document}